\documentclass[11pt,english,letterpaper]{amsart}

\pdfoutput=1

\usepackage[utf8]{inputenc}
\usepackage[T1]{fontenc}
\usepackage{lmodern}
\usepackage{amsthm, amssymb, amsmath, amsfonts, mathrsfs}

\usepackage{microtype}

\usepackage[pagebackref,colorlinks=true,pdfpagemode=none,urlcolor=blue,
linkcolor=blue,citecolor=blue]{hyperref}

\usepackage{amsmath,amsfonts,amssymb,amsthm}
\usepackage{amsthm, amssymb, amsmath, amsfonts, mathrsfs}

\usepackage{mathrsfs}
\usepackage{MnSymbol}
\usepackage{scalerel} 

\usepackage{color}
\usepackage{accents}

\usepackage{bbm}

\definecolor{labelkey}{gray}{.8}
\definecolor{refkey}{gray}{.8}

\definecolor{darkred}{rgb}{0.9,0.1,0.1}
\definecolor{darkgreen}{rgb}{0,0.5,0}

\newtheorem{theorem}{Theorem}[section]
\newtheorem{lemma}[theorem]{Lemma}

\newtheorem{proposition}[theorem]{Proposition}

\theoremstyle{remark}
\newtheorem{remark}[theorem]{Remark}

\renewenvironment{proof}[1][Proof]{ {\itshape \noindent {#1.}} }{$\Box$
\medskip}

\numberwithin{equation}{section}
\numberwithin{figure}{section}
\theoremstyle{plain}

\newcommand{\R}{\mathbb{R}}

\newcommand{\1}{\mathbbm{1}}
\newcommand{\eps}{\varepsilon}
\newcommand{\bV}{\mathcal{V}}
\newcommand{\bX}{\mathrm{X}}
\newcommand{\bY}{\mathrm{Y}}
\newcommand{\malD}{\mathbb{D}}
\newcommand{\dd}{\mathrm{d}}
\newcommand{\lipc}{\sigma_{\mathrm{Lip}}}
\newcommand{\Var}{\mathrm{Var}}
\newcommand{\Cov}{\mathrm{Cov}}
\newcommand{\filt}{\mathscr{F}}
\newcommand{\Prob}{\mathbb{P}}
\newcommand{\hilb}{\mathcal{H}}

\begin{document}
\title[Fluctuations of Nonlinear SHE in $d=2$]{Gaussian fluctuations of a nonlinear stochastic heat equation in dimension two}

\author{Ran Tao}

\address[Ran Tao]{Department of Mathematics, University of Maryland, College Park, MD, 20740.  rantao16@umd.edu}

\begin{abstract}
    We study the Gaussian fluctuations of a nonlinear stochastic heat equation in spatial dimension two. The equation is driven by a Gaussian multiplicative noise. The noise is white in time, smoothed in space at scale $\eps$, and tuned logarithmically by a factor $\frac{1}{\sqrt{\log \eps^{-1}}}$ in its strength. We prove that, after centering and rescaling, the solution random field converges in distribution to an Edwards-Wilkinson limit as $\eps \downarrow 0$. The tool we used here is the Malliavin-Stein's method. We also give a functional version of this result.
\bigskip


\end{abstract}
\maketitle

\section{Introduction}
\label{s.introduction}
\subsection{Main result}
Consider the following parameterized two-dimensional nonlinear stochastic heat equation (SHE) with constant initial data on $(t,x) \in [0,+\infty)\times \R^2$:
\begin{align}
\partial_t u^\eps (t,x) &= \frac{1}{2} \Delta u^\eps (t,x) + \frac{\beta}{\sqrt{\log \eps^{-1}}}\sigma(u^\eps(t,x))\dd W_{\phi^\eps}(t,x);
    \label{eq:she}\\
    u^\eps(0,x) &= 1.\label{eq:ic}
\end{align}
Here $\beta>0$ is a constant; $\sigma:[0,\infty)\to[0,\infty)$ is a globally Lipschitz function satisfying $\sigma(0)=0$, $\sigma(1) \neq 0$, and $|\sigma(x)-\sigma(y)|\leq \lipc|x-y|$ for all $x,y \in \R^2$ with $\lipc>0$ fixed. 
We define 
\begin{equation}
     \dd W_{\phi^\eps}(t,x) = \phi^\eps * \dd W (t,x), \label{eq:colorednoise}
\end{equation} where $\dd W(t,x)$ is a space-time white noise on $[0,+\infty)\times\R^2$, $\phi \in C_c^\infty(\R^2)$ is a non-negative mollifier with $\int{\phi} \dd x =1$ and $\phi^\eps (x) = \frac{1}{\eps^2}\phi(\frac{x}{\eps})$, and $*$ denotes convolution in space.  

The noise $\dd W_{\phi^\eps}(t,x)$ is a centered Gaussian noise, white in time and homogeneously colored in space. The spatial correlation length is at scale $\eps$. Formally, the covariance operator
of $\dd W_{\phi}^{\eps}$ is given by
\begin{equation}
	\mathbb{E}\left[\dd W_{\phi}^{\eps}(t,x)\dd W_{\phi}^{\eps}(t',x')\right]=\delta_0(t-t')\tfrac{1}{\eps^2}R(\tfrac{x-x'}{\eps}).\nonumber\label{eq:covarianceoperator}
\end{equation}
Here $\delta_0$ is the Dirac delta measure with unit mass at zero. $R(x)$ is non-negative and non-negative definite, given by
\begin{equation}
    R(x)= \int_{\R^2} \phi(x+y)\phi(y)\dd y \in C_c^\infty(\R^2).\nonumber
    \label{eq:functionR}
\end{equation}
For $\eps>0$, it is well-known that the initial value problem \eqref{eq:she}--\eqref{eq:ic} is well-posed and has a mild formulation
\begin{equation}
    u^\eps (t,x) =1+\frac{\beta}{\sqrt{\log \eps^{-1}}} \int_{0}^{t} \int_{\R^2} G_{t-s}(x-y) \sigma\left(u^{\eps}(s, y)\right) \dd W_{\phi^\eps}(s, y).\label{eq:umild}
\end{equation}
Here $G_{t}(x):=\frac{1}{2\pi t}\mathrm{e}^{-|x|^{2}/(2t)}$
denotes the two-dimensional heat kernel. The stochastic integral in \eqref{eq:umild} is interpreted in the Itô-Walsh sense.

Our main result is the following central limit theorem (CLT):
\begin{theorem}\label{thm:mainthm}
There exists some $\beta_0 \in (0, \frac{\sqrt{2\pi}}{\lipc})$ such that if $\beta <\beta_0$, for any fixed $T>0$ and any fixed Schwartz function $g\in C_c^\infty(\R^2)$, the random variable
\begin{equation}
    \bX^{\eps,T}(g):=\sqrt{\log \eps^{-1}} \int_{\R^2}\left[u^{\eps}(T, x)-1\right] g(x) \dd x
    \label{eq:bigX}
\end{equation}
converges in law to a Gaussian distribution
\begin{equation}
    \bX^T(g):=\int_{\R^2} U(T, x) g(x) \dd x
    \label{eq:bigXlimit}
\end{equation}
as $\eps \to 0$. Here $U$ is the random distribution that solves the Edwards-Wilkinson equation in dimension two:
\begin{equation}
    \partial_t U=\frac{1}{2} \Delta U+\beta\sqrt{\mathbb{E} \sigma\left(\Xi_{1, 2}(2)\right)^{2}} \dd W(t, x), \qquad U(0, x)=0.\label{eq:edwardswilkinson}
\end{equation}
$\Xi_{1,2}(\cdot)$ is a random process defined by the following as in \cite[p.3]{dunlap2020forward}:\\
Let $\{B(q)\}_{q\ge0}$ be a 1D standard Brownian motion with the natural filtration
$\{\mathcal{G}_{q}\}_{q\ge0}$. Then $\Xi_{1,2}(\cdot)$ is the (unique) solution to a forward-backward stochastic differential equation (FBSDE):
\begin{align}
    \dd \Xi_{1,2}(q) & =\frac{\beta}{2\sqrt{\pi}}\big(\mathbb{E}[\sigma^2(\Xi_{1,2}(2))\mid\mathcal{G}_{q}]\big)^{1/2}\dd B(q),\qquad q\in(0,2];\label{eq:fbsde}\\
	\Xi_{1,2}(0)     & =1.\label{eq:fbsdeinitialcon}
\end{align}
\end{theorem}
The existence and uniqueness of solutions to \eqref{eq:fbsde}--\eqref{eq:fbsdeinitialcon} are given in \cite[Theorem~1.1]{dunlap2020forward}. To keep the consistency of notations, the subscript $_{1,2}$ denotes that the initial data equals to $1$ constantly and the terminal time is at $q=2$.

From \cite[Theorem~1.1]{dunlap2020forward}, we have that the coefficient in \eqref{eq:edwardswilkinson} satisfies $$\beta\sqrt{\mathbb{E} \sigma\left(\Xi_{1, 2}(2)\right)^{2}} = 2 \sqrt{\pi} J(2,1),$$ where $J^2(q,a)$ is a viscosity solution to the quasilinear heat equation
\begin{align*}
    \partial_q J^2 &= \frac{1}{2} J^2\partial_{aa}J^2;\\
    J^2(0,a) &= \frac{\beta^2}{4\pi}\sigma^2(a).
\end{align*}

We also prove the following functional version of the CLT:
\begin{theorem}\label{thm:functionalCLT}
Let $\beta <\beta_0$. For any $T>0$ and $g\in C_c^\infty(\R^2)$. 
We have convergence in law in the space of continuous functions $C([0,T])$,
\begin{equation}
    \left(\bX^{\eps,t}(g) \right)_{t \in [0,T]}\to \left(\beta\sqrt{\mathbb{E} \sigma\left(\Xi_{1, 2}(2)\right)^{2}}  \int_0^t \int_{\R^2} G_{t-s}*g(x)\dd W(s,x) \right)_{t\in [0,T]}, \nonumber
\end{equation}
as $\eps \to 0$.
As above, $\dd W$ is a space-time white noise on $[0,+\infty)\times \R^2$.
\end{theorem}

\subsection{Context}
The Gaussian fluctuations of many SHE and KPZ equations have been widely studied. 

If $\sigma(x)=x$, our SHE \eqref{eq:she}--\eqref{eq:ic} becomes a linear one:
\begin{align}
\partial_t v^\eps (t,x) &= \frac{1}{2} \Delta v^\eps (t,x) + \frac{\beta}{\sqrt{\log \eps^{-1}}}v^\eps(t,x)\dd W_{\phi^\eps}(t,x);
    \label{eq:shelinear}\\
    v^\eps(0,x) &= 1.\label{eq:iclinear}
\end{align}
It is known that there is a phase transition at $\beta = \sqrt{2\pi}$ for this multiplicative SHE in dimension two. 
In \cite{caravenna2017universality}, Caravenna, Sun, and Zygouras showed that for any fixed $T>0$ and $x\in \R^2$, if $\beta \geq \sqrt{2\pi}$,  $v^{\eps}(T,x)$ converges to 0 in probability as $\eps \to 0$; if $\beta<\sqrt{2\pi}$ (known as $\beta$ in the "subcritical" regime),  $v^{\eps}(T,x)$ converges in law to a log-normal random
variable as $\eps \to 0$. They also proved that (see \cite[Theorem~2.17]{caravenna2017universality}) if $\beta<\sqrt{2\pi}$, for any fixed time $T>0$ and any fixed Schwartz function $g\in C_c^\infty(\R^2)$, the random variable
\begin{equation}
    \sqrt{\log \eps^{-1}} \int_{\R^2}\left[v^{\eps}(T, x)-1\right] g(x) \dd x 
    \nonumber
\end{equation}
converges in law to a Gaussian distribution. In other words, the solution random field to \eqref{eq:shelinear}--\eqref{eq:iclinear}, after centering and rescaling, converges (in distribution) to the solution to an Edwards-Wilkinson equation.

We are interested in obtaining an analog result for our nonlinear SHE \eqref{eq:she}--\eqref{eq:ic}. In \cite{caravenna2017universality}, the authors used the Feynman–Kac formula, which is not available for nonlinear case. Therefore, new methods should be used for this problem. In \cite{dunlap2020forward}, Dunlap and Gu characterized the local statistics for the limiting solution field to nonlinear SHE \eqref{eq:she}--\eqref{eq:ic} when $\beta\lipc < \sqrt{2\pi}$. (See Proposition~ \ref{pr:limitObject} below or \cite[Theorem~1.2]{dunlap2020forward}.) They showed that $u^{\eps}(T,x)$ converges in law to a specific random variable $\Xi_{1,2}(2)$ as $\eps \to 0$. They also gave the multipoint statistics in their work.

Based on their findings, we want to study the asymptotics of the (centered and rescaled) solution field to the nonlinear SHE as $\eps \to 0$. One may conjecture an Edwards-Wilkinson limit as in the linear case. In Theorem~\ref{thm:mainthm}, we gave a proof to this limit for a regime $\beta<\beta_0<\frac{\sqrt{2\pi}}{\lipc}$. Whether or not this result can be extended to the entire "subcritical" regime $\beta \in (0, \frac{\sqrt{2\pi}}{\lipc})$ remains unknown. We will discuss it in more detail below.

We proved Theorem~\ref{thm:mainthm} by using the mild formulation \eqref{eq:umild} and the Malliavin-Stein's method. The use of the Malliavin-Stein's method in proving the Gaussian fluctuations for SHE is inspired by works \cite{chen2021central, gu2020nonlinear, huang2020central, huang2020gaussian, nualart2020averaging}. We also use a multivariate version to prove a functional CLT in Theorem~\ref{thm:functionalCLT}. 

Through a Hopf-Cole transformation $h=\log u$, the linear SHE \eqref{eq:shelinear}--\eqref{eq:iclinear} is related to a KPZ-type equation in $d=2$. The Gaussian fluctuations of the KPZ equation in $d=2$ in the subcritical regime is proved in \cite{caravenna2020two} and \cite{gu2020gaussian}. In  \cite{nakajima2021fluctuations}, the authors proved the Gaussian fluctuations of the linear SHE and KPZ equations in dimension two with general initial conditions. One of the reasons why we are interested in the nonlinear SHE is that we hope it will shed lights on the study of more general Hamilton–Jacobi SPDEs:
\begin{equation}
    \partial_t h(t,x) = \frac{1}{2} \Delta h(t,x) + H(\nabla h(t,x) ) + \beta \dd W_{\phi}(t,x), \label{eq:hjkpz}
\end{equation}
where the Hamiltonian $H$ is not necessarily quadratic. The only result in this direction that we are aware of is the study of a two-dimensional anisotropic KPZ equation in \cite{cannizzaro20212d}. While the nonlinear SHE \eqref{eq:she}--\eqref{eq:ic} is more complicated than the linear one, it is still much more approachable than \eqref{eq:hjkpz} because we can make use of the mild formulation \eqref{eq:umild}.

\cite{gu2020nonlinear} presents analog Gaussian fluctuations of the nonlinear SHE in $d\geq 3$. Gaussian fluctuations of the linear SHE and KPZ equations in $d\geq 3$ are studied in \cite{comets2019space, cosco2020law, dunlap2020fluctuations, gu2018edwards} and \cite{magnen2018scaling}. 

There are a few other things we would like to remark before we proceed. 

\subsubsection{$\beta$-region}
In general, for SHE with multiplicative noise and KPZ equations in $d\geq 2$, the noise-strength parameter $\beta$ plays a noticeable role. 

For linear SHE \eqref{eq:shelinear}--\eqref{eq:iclinear}, if $v^{\eps}(t,x)$ converges to 0 in probability as $\eps \to 0$ for any $t>0, x \in \R^2$,  we say that the system is in the \textit{strong disorder} regime. If $v^{\eps}(t,x)$ converges in distribution to a non-degenerate random variable as $\eps \to 0$ for any $t>0, x \in \R^2$, we say that the system is in the \textit{weak disorder} or \textit{subcritical} regime. 

As stated above, in \cite{caravenna2017universality}, the authors proved that $\beta_c = \sqrt{2\pi}$ is the exact critical threshold where the departure from weak disorder to strong disorder takes place. They also proved the Edwards-Wilkinson limit for linear SHE in the entire subcritical regime $\beta \in (0, \sqrt{2\pi})$. (Work \cite{gu2020gaussian} gives an alternative proof to this limiting result for a smaller regime $\beta<\beta_0$, where $\beta_0<\sqrt{2\pi}$.) When $\beta \approx \sqrt{2\pi}$, the solution random field $v^{\eps}(t,x)$ no longer needs any centering and rescaling to obtain a nontrivial limiting structure. We refer to \cite{caravenna2021critical, gu2021moments, caravenna2019moments} for the study in this direction. 

Here we are to prove the Edwards-Wilkinson limit (Theorem~\ref{thm:mainthm}) for nonlinear SHE \eqref{eq:she}--\eqref{eq:ic} in a regime $\beta<\beta_0<\frac{\sqrt{2\pi}}{\lipc}$ smaller than the entire subcritical regime $\beta \in (0, \frac{\sqrt{2\pi}}{\lipc})$. In fact, our proof below ensures that Theorem~\ref{thm:mainthm} would hold for  $\beta_0 = \frac{1}{2\sqrt{6}}\frac{\sqrt{2\pi}}{\lipc}$ (see Remark~\ref{re:threshold} for detail). While we don't believe that this is an optimal $\beta$-region, whether or not (and how) Theorem~\ref{thm:mainthm} can be extended to the entire subcritical regime remains open. 

For SHE in $d \geq 3$, the readers may refer to \cite{mukherjee2016weak} and \cite{cosco2020law} on the weak/strong disorder regions of $\beta$. In particular, \cite{cosco2020law} proved the Edwards-Wilkinson limit for linear SHE in $d\geq 3$ in the entire subcritical regime. For nonlinear SHE in $d\geq 3$,  \cite{gu2020nonlinear} also restricted $\beta$ in a smaller region.

\subsubsection{General initial condition}
Following the same methods as in \cite{dunlap2020forward}, we can generate the results from \cite[Theorem~1.2]{dunlap2020forward} to nonlinear SHE \eqref{eq:she} with general initial condition 
\begin{equation}
    u^{\eps}(0,x)=u_0(x),  \qquad x\in \R^2,\label{eq:genic}
\end{equation}
if we assume $u_0$ satisfies
\begin{equation*}
    0<\inf _{x \in \mathbb{R}^{2}} u_{0}(x) \leq \sup _{x \in \mathbb{R}^{2}} u_{0}(x)<+ \infty. 
\end{equation*}
In fact, it can be shown that $u^{\eps}(T,x) \to \bar{u}(T,x)\Xi_{1,2}(2)$ in law as $\eps \to 0$ where $\bar{u}(T,x) = G_{T}*u_0(x)$ is a deterministic function. One can now go through the exactly same arguments as in the proof of Theorem~\ref{thm:mainthm} here to obtain the Gaussian fluctuations of the SHE \eqref{eq:she} with initial condition \eqref{eq:genic} as $\eps \to 0$. We claim, without proof, that Theorem~\ref{thm:mainthm} still holds, except that the Edwards-Wilkinson equation \eqref{eq:edwardswilkinson} should be replaced by a new SPDE:
\begin{equation*}
     \partial_t U=\frac{1}{2} \Delta U+\beta\sqrt{\mathbb{E} \sigma\left(\Xi_{1, 2}(2)\right)^{2}} \bar{u}(t,x)\dd W(t, x), \qquad U(0, x)=0. 
\end{equation*}
This result aligns with the one in \cite[Theorem~1.3, Remark~1.4]{nakajima2021fluctuations} for linear SHE in dimension two with general initial condition.

\subsection{Organization of the paper}
In Section~\ref{s.pre}, we first scale our SHE \eqref{eq:she}--\eqref{eq:ic} to microscopic variables. Then we introduce the basics of the Malliavin-Stein's method and prove some uniform moment bounds to use in the sequel. We will also state \cite[Theorem~1.2]{dunlap2020forward} in Proposition~\ref{pr:limitObject} for convenience. In Section~\ref{se:mainproof} and \ref{se:clt}, we prove Theorem~\ref{thm:mainthm} and \ref{thm:functionalCLT} respectively.
\subsection*{Acknowledgement}
The author would like to thank her advisor, Yu Gu, for suggesting this problem and providing guidance. She would also like to thank Alex Dunlap for comments and discussion. This work is supported by Yu Gu’s
NSF grant DMS-2203007.

\section{Preliminaries}
\label{s.pre}

\subsection{Scaling}
\label{se:scaling}
We are going to use the Malliavin calculus for Gaussian spaces. In order to simplify the calculation, we want to fix a specific Gaussian space. For this purpose, we introduce another representation of the SHE  \eqref{eq:she}--\eqref{eq:ic}.

We study the following equation in microscopic variables on $(t,x)\in[0,+\infty)\times \R^2$:
\begin{align}
    \partial_{t} \bV^{\eps}(t, x)&=\frac{1}{2} \Delta \bV^{\eps}(t, x)+\frac{\beta}{\sqrt{\log \eps^{-1}}} \sigma\left(\bV^{\eps}(t, x)\right) \dd W_{\phi}(t, x);\label{eq:microscopicSHE}\\
    \bV^{\eps}(0, x)&=1. \label{eq:microscopicSHEic}
\end{align}
Here $ W_{\phi}(t,x) := W_{\phi^1}(t,x) = \int_{\R^2}\phi(x-y)W(t,y)dy$ is defined as in \eqref{eq:colorednoise} with $\eps =1$.
By the scaling property of the space-time white noise, we have
\begin{equation}
    u^\eps (\cdot,\cdot) = \bV^\eps(\frac{\cdot}{\eps^2}, \frac{\cdot}{\eps}) \quad \text{jointly in law}.\nonumber
    \label{eq:timescaleid}
\end{equation}

Since we are only interested in proving convergence in law, it is sufficient to prove Theorem~\ref{thm:mainthm} with $u^\eps(T,x)$ substituted by $\bV^\eps(\frac{T}{\eps^2}, \frac{x}{\eps})$. In fact, as the mild formulation of $\bV^\eps(t,x)$ is
\begin{equation}
    \bV^{\eps}(t,x)
    = 1+\frac{\beta}{\sqrt{\log \eps^{-1}}} \int_{0}^{t} \int_{\R^2} G_{t-s}(x-y) \sigma\left(\bV^{\eps}(s, y)\right) \dd W_{\phi}(s, y),
    \label{eq:bigVmild}
\end{equation}
we have
\begin{equation}
\begin{split}
    &\bX^{\eps,T}(g)\stackrel{\text { law }}{=}  \bY^{\eps,T}(g) :=\sqrt{\log \eps^{-1}} \int_{\R^2}\left[\bV^{\eps}(\frac{T}{\eps^2}, \frac{x}{\eps})-1\right] g(x) \dd x\\
    &= \beta \int_{\R^2} \int_{0}^{\frac{T}{\eps^2}} \int_{\R^2} G_{\frac{T}{\eps^2}-s}(\frac{x}{\eps}-y) \sigma\left(\bV^{\eps}(s, y)\right) \dd W_{\phi}(s, y)g(x) \dd x.\label{eq:bigY}
\end{split}
\end{equation}
For $0\leq t_1 <\dots<t_m \leq T$, we have \begin{equation*}
    \left(X^{\eps, t_1}(g), \dots, X^{\eps, t_m}(g)\right) \stackrel{\text { law }}{=} \left(Y^{\eps, t_1}(g), \dots, Y^{\eps, t_m}(g)\right).\label{eq:multidimeqinlaw}
\end{equation*}
From now on, we will study $Y^{\eps,T}(g)$ instead of $X^{\eps,T}(g)$, etc.
\begin{remark}
As discussed in \cite[Section 1.1]{dunlap2020forward}, in the subcritical regime $\beta \lipc <\sqrt{2\pi}$, the 2-dimensional stochastic heat equation \eqref{eq:microscopicSHE}--\eqref{eq:microscopicSHEic} evolves on an exponential time scale with respect to the strength of the random noise. The exponential time scale explains why we need the tuning $\frac{1}{\sqrt{\log \eps^{-1}}}$ before the noise term in our SHE \eqref{eq:she}--\eqref{eq:ic}.

This scaling is different from the $d \geq 3$ cases. When $d \geq 3$, if the strength of the noise is small enough so that the system is in the weak disorder regime, the SHE \eqref{eq:microscopicSHE}--\eqref{eq:microscopicSHEic} evolves in an 'arbitrarily long' diffusion scale $(\frac{t}{\lambda},\frac{x}{\lambda^2})$ where $\lambda$ is independent of the strength of the noise. See e.g. \cite{mukherjee2016weak}.
\end{remark}

\subsection{The Malliavin calculus}
Now we state the Gaussian space in which we will be working with and give some elementary results from the Malliavin calculus theory. A detailed discussion should be referred to the monograph \cite{nualart2006malliavin}.

Consider a complete probability space $(\Omega, \filt, \Prob)$ generated by the Gaussian noise $\dd W_{\phi}$ as defined above. Let $\hilb$ be the Hilbert space given by the closure of $C_c^{\infty}(\R_{\geq 0} \times \R^2)$ with respect to the inner product 
\begin{equation}
    \langle f,g \rangle_{\hilb} := \int_0^{+\infty} \int_{\R^2} \int_{\R^2} f(s,x)g(s,y)R(x-y) \dd x \dd y \dd s.\nonumber
    \label{eq:innerproductHilbert}
\end{equation}
Define an operator $W: \hilb \to L^2(\Omega)$ as
\begin{equation}
W(h) = \int_0^{+\infty}\int_{\R^2}h(t,x)\dd W_{\phi}(t,x).\nonumber
    \label{eq:isometry}
\end{equation}
Then $W$ is a linear isometry and the Gaussian space $\{W(h): h\in \hilb\}$ is an isonormal Gaussian process over $\hilb$.

Following the conventional notations, we denote the Malliavin derivative operator for the isonormal Guassian process $\{W(h): h\in \hilb\}$ by $D$. Let $\malD^{k,p}, k,p \in \mathbb{Z}^+$ denote the associated Gaussian Sobolev spaces.

For any $t\geq0$, let $\filt_t$ be the $\sigma$-algebra generated by $\{W_{\phi}([0,s]\times A): 0 \leq s \leq t, A\in\mathscr{B}_b(\R^2)\}$ and the null sets of $\filt$. Here $\mathscr{B}_b(\R^2)$ denotes the bounded Borel subsets of $\R^2$. A random field $v = \{v(t,x), t\geq 0, x\in \R^2\}$ is adapted if it is $\filt_t$-measurable for each $t\geq 0$ and $x\in \R^2$.

Let $\delta$ denote the divergence operator. $\delta$ is the adjoint operator of $D$. 
If $v(t,x)$ is an adapted and measurable random field on $[0,+\infty)\times \R^2$ such that 
\begin{equation}
\int_0^{+\infty}\int_{\R^2}\int_{\R^2}\mathbb{E}\left(v(t,x)v(t,y)\right)R(x-y)\dd x \dd y\dd t <+\infty, \nonumber
\label{eq:WalshIntegraliso}
\end{equation}
then $v \in \mathrm{Dom}(\delta)$ and 
\begin{equation}
    \delta(v)=\int_0^{+\infty}\int_{\R^2} v(t,x)\dd W_{\phi}(t,x),
    \label{eq:Skorokhodidentity}\nonumber
\end{equation}
where the last integral is well-defined in the Itô-Walsh sense with the isometry
\begin{equation}
    \mathbb{E}\left(\delta(v)\right)^2 = \int_0^{+\infty}\int_{\R^2}\int_{\R^2}\mathbb{E}\left(v(t,x)v(t,y)\right)R(x-y)\dd x \dd y\dd t 
    \label{eq:itoisometry}\nonumber
\end{equation}
By \cite[Proposition 1.3.8]{nualart2006malliavin}, if $v\in \mathrm{Dom}(\delta)$ and is an adapted process in $\malD^{1,2}$, then 
\begin{equation}
D_{r,z}\delta(v)=v(r,z)+\int^{+\infty}_0\int_{\R^2}D_{r,z}v(s,y)\dd W_\phi(s,y).\nonumber
    \label{eq:commutation}
\end{equation}
 A standard Picard's iteration scheme shows that for all $\eps>0$ and $(t, x) \in[0, +\infty) \times \R^{2}$, $\bV^\eps(t,x) \in \malD^{1, p}$ for all $p \geq 1$. Since $\bV^\eps(t,x)$ has the mild formulation \eqref{eq:bigVmild}, the Malliavin derivative $D_{r,z}\bV^\eps(t,x)$ satisfies the equation
\begin{align}
    D_{r,z}\bV^\eps(t,x) &= \frac{\beta}{\sqrt{\log \eps^{-1}}}\1_{[0,t]}(r)G_{t-r}(x-z)\sigma(\bV^\eps(r,z)) \nonumber\\
    &+ \frac{\beta}{\sqrt{\log \eps^{-1}}}\int_r^t\int_{\R^2}G_{t-s}(x-y)D_{r,z}\sigma(\bV^\eps(s,y))\dd W_\phi(s,y) \nonumber\\
    &= \frac{\beta}{\sqrt{\log \eps^{-1}}}\1_{[0,t]}(r)G_{t-r}(x-z)\sigma(\bV^\eps(r,z)) \nonumber\\
    &+ \frac{\beta}{\sqrt{\log \eps^{-1}}}\int_r^t\int_{\R^2}G_{t-s}(x-y)\Sigma(s,y)D_{r,z}\bV^\eps(s,y)\dd W_\phi(s,y).
    \label{eq:malDofbigV}
\end{align}
Here $\Sigma$ is an adapted process bounded by the Lipschitz constant $\lipc$. If we further assume that $\sigma \in C^1([0,+\infty))$, then $\Sigma(s,y) = \sigma'(\bV^\eps(s,y))$.

In addition, we will need the following version of the Clark-Ocone formula in the proof of Theorem~\ref{thm:mainthm}.
\begin{proposition}[Clark-Ocone Formula]
For $X \in \malD^{1,2}$,
\begin{equation}
    X = \mathbb{E}X+\int_0^{+\infty}\int_{\R^2}\mathbb{E}(D_{r,z}X|\filt_r)\dd W_{\phi}(r,z) \qquad \Prob-\text{almost surely}. \nonumber
\end{equation}\label{pr:clark}
\end{proposition}
\begin{proof}
See \cite[Proposition 6.3]{chen2021spatial}.
\end{proof}

\subsection{Moment bounds}
We will use the following two uniform moment bounds in the proof of Theorem~\ref{thm:mainthm}.

Note that in both lemmas we are not able to cover the entire subcritical regime $0<\beta<\frac{\sqrt{2\pi}}{\lipc}$. In fact, the extension of Lemma~\ref{le:Vmomentbound} to the entire subcritical regime should be true and its counterpart in the 2D directed polymer environment is proved in \cite{lygkonis2021}. However, it remains unknown whether  Lemma~\ref{le:MallivianMomentbound} can be extended to the entire subcritical regime.

We first give a uniform moment bound on the mild solution to the SHE.
\begin{lemma}\label{le:Vmomentbound}
For all $p\geq 2$, there exists $0<\beta_0(p) \leq \frac{\sqrt{2\pi}}{\lipc} $ depending only on $p$, such that for all $T>0$ and $\beta < \beta_0(p)$, there exists $\eps_0>0$ (depending on $\phi, T$, $\beta\lipc$ and $p$) and $C>0$ (depending on $\beta\lipc$ and $p$) such that
\begin{equation}
\sup_{t\in [0,T], x \in \R^2}\mathbb{E}\left|u^\eps(t,x)\right|^p\leq C^p, \label{eq:umomentbound}
\end{equation}
 for all $\eps<\eps_0$. 
\end{lemma}
\begin{proof}
First we notice that, for any $\eps, t>0$, $u^\eps(t,\cdot)$ is stationary in the $x$-variable, so we only need to prove for $u^\eps(t,0)$. 

By the mild formulation \eqref{eq:umild}, the triangle inequality and the inequality $2ab\leq (ca)^2+(\frac{b}{c})^2$ for all $c>1$, we have that for any arbitrary $\alpha>1$,
\begin{align}
    &\left(\mathbb{E}\left|u^\eps(t,0)\right|^p\right)^{\frac{2}{p}} \nonumber \\
    &\leq \frac{\alpha}{\alpha-1}+\frac{\alpha\beta^2}{\log \eps^{-1}}\left( \mathbb{E}\left(\int_{0}^{t} \int_{\R^2} G_{t-s}(y) \sigma\left(u^{\eps}(s, y)\right) \dd W_{\phi^\eps}(s, y)\right)^p\right)^{\frac{2}{p}} \nonumber
\end{align}
Apply the Burkholder-Davis-Gundy inequality, we have that with $c_p = \frac{p(p-1)}{2}$,
\begin{align}
    &\left(\mathbb{E}\left|u^\eps(t,0)\right|^p\right)^{\frac{2}{p}} \nonumber\\
    &\leq \frac{\alpha}{\alpha-1} +c_p\frac{\alpha\beta^2}{\log \eps^{-1}}\nonumber \\
    &\cdot \int_{0}^{t}\left(\mathbb{E}\left( \int_{\R^2}\int_{\R^2} \left|\prod_{i=1}^2 G_{t-s}(y_i) \sigma(u^{\eps}(s, y_i))\right|\frac{1}{\eps^2}R(\frac{y_1-y_2}{\eps}) \dd y_1 \dd y_2 \right)^{\frac{p}{2}}\right)^{\frac{2}{p}} \dd s\nonumber
\end{align}
The H\"older's inequality gives that
\begin{align}\label{eq:uholder}
    \left(\mathbb{E}\left|\sigma(u^\eps(t,x))\sigma(u^\eps(t,y))\right|^{\frac{p}{2}} \right)^{\frac{2}{p}}&\leq \left(\mathbb{E}\left|\sigma(u^\eps(t,x))\right|^p\mathbb{E}\left|\sigma(u^\eps(t,y))\right|^p\right)^{1/p}\\ &= (\mathbb{E}\left|\sigma(u^\eps(t,x))\right|^p)^{\frac{2}{p}}\leq \lipc^2\left(\mathbb{E}\left|u^\eps(t,0)\right|^p\right)^{\frac{2}{p}} \nonumber
\end{align}
Combined the above results. Use the Minkowski inequality and a change of variable $y_1-y_2 \mapsto y$, we obtain
\begin{align*}
    &\left(\mathbb{E}\left|u^\eps(t,0)\right|^p\right)^{\frac{2}{p}}  \\
    & \leq \frac{\alpha}{\alpha-1} +c_p\frac{\alpha\beta^2}{\log \eps^{-1}}  \int_{0}^{t} \int_{\R^2} \int_{\R^2} \left(\mathbb{E}\left|\prod_{i=1}^{2} G_{t-s}(y_i) \sigma\left(u^{\eps}(s, y_i)\right)\right|^{\frac{p}{2}}\right)^{\frac{2}{p}}\frac{1}{\eps^2}R(\frac{y_1-y_2}{\eps})\dd y_1 \dd y_2 \dd s \\
    & \leq \frac{\alpha}{\alpha-1} +c_p\frac{\alpha\beta^2\lipc^2}{\log \eps^{-1}}  \int_{0}^{t} \int_{\R^2} \int_{\R^2}
    G_{t-s}(y_1)G_{t-s}(y_2)\left(\mathbb{E}\left|u^\eps(t,0)\right|^p\right)^{\frac{2}{p}}\frac{1}{\eps^2}R(\frac{y_1-y_2}{\eps})\dd y_1 \dd y_2 \dd s \\
    &= \frac{\alpha}{\alpha-1} +c_p\frac{\alpha\beta^2\lipc^2}{\log \eps^{-1}} \int_{0}^{t} \int_{\R^2} G_{2(t-s)}(y)\left(\mathbb{E}\left|u^\eps(s,0)\right|^p\right)^{\frac{2}{p}}\frac{1}{\eps^2}R(\frac{y}{\eps})\dd y \dd s.
\end{align*}
It is easy to see that with $R(\cdot) \leq \|\phi\|_{\infty}$ and $\int_{\R^2}\frac{1}{\eps^2}R(\frac{y}{\eps})dy=1$, we have the elementary inequality \begin{equation*}
    \int_{\R^2} G_{2(t-s)}(y)\frac{1}{\eps^2}R(\frac{y}{\eps})\dd y \leq \frac{1}{4\pi(t-s)}\wedge \frac{\|\phi\|_{\infty}}{\eps^2}.
\end{equation*}
By the Lemma~\ref{lemma:2mobdlemma} below, if for some $\alpha>1$, \begin{equation}
    c_p \alpha \beta^2 \lipc^2 < 2 \pi,\label{eq:suffconlp} 
\end{equation} there exists an $\eps_0>0$ and $C_0=C_0(c_p, \alpha, \beta \lipc)>0$ such that for all $\eps <\eps_0$
\begin{equation}
    \left(\mathbb{E}\left|u^\eps(t,0)\right|^p\right)^{\frac{2}{p}} \leq \frac{\alpha}{\alpha-1}C_0.\nonumber
\end{equation}
Since $\alpha>1$ is arbitrary, condition \eqref{eq:suffconlp} is valid if and only if 
\begin{equation*}
    \beta<\beta_0(p)=\frac{\sqrt{2\pi}}{\sqrt{c_p}\lipc}.
\end{equation*}
In particular, $\beta_0(2) = \frac{\sqrt{2\pi}}{\lipc} $ and $\beta_0(4) = \frac{1}{\sqrt{6}}\frac{\sqrt{2\pi}}{\lipc}$. We can choose an appropriate $\alpha$ according to $p$ and $\beta\lipc$.
\label{lemma:2momentbound}
\end{proof}

As an immediate corollary, we have that under the same condition as in Lemma~\ref{le:Vmomentbound},
\begin{equation}
\sup_{t\in [0,T], x \in \R^2}\mathbb{E}\left|\bV^\eps(\frac{t}{\eps^2},\frac{x}{\eps})\right|^p\leq C^p \quad \text{for all } \eps<\eps_0. \label{eq:bigVmobd}
\end{equation}

\begin{lemma}
Fix constants $0<b<\sqrt{2\pi}$ and $a, c,T>0$. There exists an $\eps_0>0$ (depending on $b, c$ and $T$), such that for all $\eps \in (0,\eps_0)$ the following holds:

If $f^{\eps}:[0,T] \to [0,\infty)$ is a function such that for all $t\in [0,T]$,
\begin{equation*}
    f^{\eps}(t) \leq a+\frac{b^2}{\log \eps^{-1}}\int_{0}^{t} f^{\eps}(s) \left(\frac{1}{4\pi(t-s)}\wedge \frac{c}{\eps^2}\right) \dd s,
\end{equation*}
then for all $t\in [0,T]$, we have 
\begin{equation}
    f^{\eps}(t)\leq aC,
    \label{eq:boundofiteratedintegral}
\end{equation}
for some uniform constant $C>0$ depending only on $b$.
\label{lemma:2mobdlemma}
\end{lemma}
\begin{proof}
Define $[0,t]_<^j=\{(s_{1},\ldots,s_{j})\in[0,t]^{j}\mid s_{1}\le\cdots\le s_{j}\}$, we have
\begin{align}
		f^{\eps}(t) & \leq a\sum_{j=0}^{\infty}\frac{b^{2j}}{(4\pi\log\eps^{-1})^{j}}\int_{[0,t]_<^j}\prod_{k=1}^{j}\left(\frac{1}{s_{k+1}-s_{k}}\wedge \frac{4\pi c}{\eps^2}\right)\dd s_{1}\cdots\dd s_{j}\nonumber\\
		&\leq  a\sum_{j=0}^{\infty}\frac{b^{2j}}{(4\pi\log\eps^{-1})^{j}}\int_{[0,t]^j}\prod_{k=1}^{j}\left(\frac{1}{r_j}\wedge \frac{4\pi c}{\eps^2}\right)\dd r_{1}\cdots\dd r_{j}\nonumber\\
		&= a\sum_{j=0}^{\infty}\frac{b^{2j}}{(4\pi\log\eps^{-1})^{j}}\left[\int_0^t \left(\frac{1}{s}\wedge \frac{4\pi c}{\eps^2}\right)\dd s\right]^j\label{eq:2mobdinequality}
\end{align}
Notice that when $t>0$,
\begin{equation*}
    \int_{0}^{t} \left(\frac{1}{s}\wedge \frac{4\pi c}{\eps^2}\right)\dd s = \int_0^{\frac{\eps^2}{4\pi c}} \frac{4\pi c}{\eps^2} \dd s+ \int_{\frac{\eps^2}{4\pi c}}^t \frac{1}{s}\dd s = \log{\frac{4\pi ct}{\eps^2}}+1,
\end{equation*}
and
\begin{equation}
\lim_{\eps \to 0} \frac{b^2}{4\pi\log\eps^{-1}}\left(\log{\frac{4\pi cT}{\eps^2}}+1\right) = \frac{b^2}{2\pi}. \nonumber
\end{equation}
Therefore,  if $b< \sqrt{2\pi}$, there exists an $\eps_0>0$ depending on $b,c$ and $T$, such that for all $\eps \in (0,\eps_0)$, the sum of the series in \eqref{eq:2mobdinequality} converges for all $t\in [0,T]$ and is uniformly bounded by some constant $C$ depending only on $b$.
\end{proof}

Next we give the uniform moment bounds for the Malliavin derivatives.
\begin{lemma}\label{le:MallivianMomentbound}
For all $p\geq 2$, there exists $0<\beta_0(p)\leq\frac{\sqrt{2\pi}}{\lipc}$ depending only on $p$, such that for all $T>0$ and $\beta<\beta_0(p)$, there 
exists $0<\eps_0'\leq\eps_0$ (depending on $\phi$, $T$, $\beta\lipc$ and $p$) and $C'>0$ (depending on $\beta \lipc$ and $p$) such that for all $(t,x)\in (0,\frac{T}{\eps^2}]\times\R^2$ and all $(r,z)\in [0,t)\times \R^2$,
\begin{equation}
\left(\mathbb{E}\left|D_{r,z}\bV^\eps(t,x)\right|^p\right)^{\frac{1}{p}}\leq \frac{C'}{\sqrt{\log \eps^{-1}}}G_{t-r}(x-z), \quad \text{ for all } \eps<\eps_0'.\nonumber \label{eq:malliavinderivativeMoment}
\end{equation}
\end{lemma} 
\begin{proof}
We follow the idea from \cite[Lemma 2.2]{gu2020nonlinear}.

Let $\alpha>1$ be arbitrary. By \eqref{eq:malDofbigV}, the triangle inequality and the inequality $2ab \leq (ca)^2+(\frac{b}{c})^2$ for all $c>0$, for all $0\leq r<t\leq \frac{T}{\eps^2}$, $x,z \in \R^2$,
\begin{align}
    &\left(\mathbb{E}\left|D_{r,z}\bV^\eps(t,x)\right|^p\right)^{\frac{2}{p}} \nonumber\\
    &\leq \frac{\alpha}{\alpha-1}\frac{\beta^2}{\log \eps^{-1}}G_{t-r}(x-z)^2\left(\mathbb{E}|\sigma(\bV^\eps(r,z))|^p\right)^{\frac{2}{p}} \nonumber\\&+\frac{\alpha\beta^2}{\log\eps^{-1}}\left(\mathbb{E}\left|\int_r^t\int_{\R^2}G_{t-s}(x-y)\Sigma(s,y)D_{r,z}\bV^\eps(s,y)\dd W_\phi(s,y)\right|^p\right)^{\frac{2}{p}} \nonumber
\end{align}
Applying the Burkholder-Davis-Gundy inequality with $c_p=\frac{p(p-1)}{2}$, we have
\begin{align}    
&\left(\mathbb{E}\left|D_{r,z}\bV^\eps(t,x)\right|^p\right)^{\frac{2}{p}}
    \leq\frac{\alpha}{\alpha-1}\frac{\beta^2\lipc^2}{\log \eps^{-1}}G_{t-r}(x-z)^2\left(\mathbb{E}|\bV^\eps(r,z)|^p\right)^{\frac{2}{p}} \nonumber\\
    &+c_p \frac{\alpha\beta^2}{\log\eps^{-1}}  \int_r^t\bigg[\mathbb{E}\bigg(\int_{\R^2}\int_{\R^2}G_{t-s}(x-y_1)G_{t-s}(x-y_2)\Sigma(s,y_1)\Sigma(s,y_2)\nonumber\\&\cdot D_{r,z}\bV^\eps(s,y_1)D_{r,z}\bV^\eps(s,y_2)R(y_1-y_2)\dd y_1 \dd y_2 \bigg)^{\frac{p}{2}}\bigg]^{\frac{2}{p}}\dd s\nonumber
\end{align}
Using the Minkowski inequality, \eqref{eq:bigVmobd} and the boundness of $\Sigma$, we further obtain
\begin{align}
&\left(\mathbb{E}\left|D_{r,z}\bV^\eps(t,x)\right|^p\right)^{\frac{2}{p}} \nonumber\\
    &\leq  \frac{\alpha}{\alpha-1}\frac{\beta^2\lipc^2C^2}{\log \eps^{-1}}G_{t-r}(x-z)^2 +c_p \frac{\alpha\beta^2\lipc^2}{\log\eps^{-1}}  \int_r^t\int_{\R^4}R(y_1-y_2)\nonumber\\
    &\cdot G_{t-s}(x-y_1)G_{t-s}(x-y_2) \left[\mathbb{E}|D_{r,z}\bV^\eps(s,y_1)|^p\right]^{\frac{1}{p}}\left[\mathbb{E}|D_{r,z}\bV^\eps(s,y_2)|^p\right]^{\frac{1}{p}}\dd y_1 \dd y_2 \dd s.
    \nonumber 
\end{align}
If we set $t=\theta+r$ and $x=\eta+z$, let \begin{equation*}
    \mathcal{K}(\theta, \eta) := \left(\mathbb{E}\left|D_{r,z}\bV^\eps(\theta+r,\eta+z)\right|^p\right)^{\frac{1}{p}} =\left(\mathbb{E}\left|D_{r,z}\bV^\eps(t,x)\right|^p\right)^{\frac{1}{p}},
    \end{equation*}
then we can rewrite the above inequality as
\begin{align}
\label{eq:bdgMalliVcov}
\mathcal{K}(\theta,\eta)^2 &\leq \frac{\alpha}{\alpha-1}\frac{\beta^2\lipc^2C^2}{\log \eps^{-1}}G_{\theta}(\eta)^2+ c_p\frac{\alpha\beta^2\lipc^2}{\log\eps^{-1}}  \int_0^{\theta}\int_{\R^4}R(y_1-y_2)\nonumber\\
&\cdot G_{\theta-s}(\eta-y_1)G_{\theta-s}(\eta-y_2)\mathcal{K}(s,y_1)\mathcal{K}(s,y_2) \dd y_1 \dd y_2 \dd s. \nonumber
\end{align}
Applying \cite[Lemma 2.2]{chen2019comparison}, we have
\begin{equation}
    \mathcal{K}(\theta,\eta)\leq \sqrt{\frac{\alpha}{\alpha-1}}\frac{\beta\lipc C}{\sqrt{\log \eps^{-1}}}G_{\theta}(\eta)\cdot H\left(\theta, 2c_p\frac{\alpha\beta^2\lipc^2}{\log\eps^{-1}} \right)^{\frac{1}{2}} \label{eq:boundofK}
\end{equation}
where $H(t,\gamma)$ is defined for all $\gamma\geq0$ as
\begin{equation*}
    H(t,\gamma)= \sum_{n=0}^{\infty} \gamma^nh_n(t).
\end{equation*}
Here $h_0(t) := 1$ and for $n\geq 1$,
\begin{equation*}
    h_n(t) := \int_0^th_{n-1}(s)\int_{\R^2}G_{t-s}(z)R(z)\dd z\dd s.
\end{equation*}
It is easy to see that with $R(\cdot) \leq \|\phi\|_{\infty}$ and $\int_{\R^2}R(y)dy=1$, we have the elementary inequality 
\begin{equation*}
    \int_{\R^2} G_{t-s}(z)R(z)\dd z \leq \frac{1}{2\pi(t-s)}\wedge \|\phi\|_{\infty} \quad \text{for all } 0\leq s \leq t.
\end{equation*}
Since $0\leq \theta \leq t \leq \frac{T}{\eps^2}$, 
\begin{align}
    h_1(\theta)&=\int_0^{\theta}\int_{\R^2} G_{\theta-s}(z)R(z)\dd z \dd s \leq \int_0^{\theta}\left( \frac{1}{2\pi(\theta-s)}\wedge \|\phi\|_{\infty}\right) \dd s \nonumber\\
    & = \int_0^{\theta} \left(\frac{1}{2\pi s} \wedge \|\phi\|_{\infty}\right)\dd s \leq \int_0^{\frac{T}{\eps^2}} \left(\frac{1}{2\pi s} \wedge \|\phi\|_{\infty}\right)\dd s  \nonumber \\
    &= \int_0^T\left(\frac{1}{2\pi s} \wedge \frac{\|\phi\|_{\infty}}{\eps^2}\right)\dd s = \int_0^{\frac{\eps^2}{2 \pi \|\phi\|_{\infty}}} \frac{\|\phi\|_{\infty}}{\eps^2} \dd s + \int_{\frac{\eps^2}{2 \pi \|\phi\|_{\infty}}}^T \frac{1}{2\pi s}\dd s\nonumber\\
    &=\frac{1}{2\pi}+ \frac{1}{2\pi}\log \frac{2 \pi \|\phi\|_{\infty}T}{\eps^2}. \nonumber
\end{align}
Thus for all $n\geq 1$,
\begin{equation}
    h_n(\theta) \leq \left[\frac{1}{2\pi}+ \frac{1}{2\pi}\log \frac{2 \pi \|\phi\|_{\infty}T}{\eps^2}\right]^n.\nonumber
\end{equation}
Notice that we have
\begin{equation}
    \lim_{\eps \to 0}\left(2c_p\frac{\alpha\beta^2\lipc^2}{\log\eps^{-1}}\right)\cdot\left(\frac{1}{2\pi}+ \frac{1}{2\pi}\log \frac{2 \pi \|\phi\|_{\infty}T}{\eps^2}\right) = \frac{2c_p \alpha \beta^2 \lipc^2}{\pi}.\label{eq:limitforsum}
\end{equation}
The limit is independent of $\phi$ and $T$.
Therefore, if for some $\alpha >1$,
\begin{equation}
    2c_p\alpha \beta^2 \lipc^2 < \pi, \label{eq:conditionforbeta}
\end{equation}
 then there exists an $\eps_0'>0$, such that when $\eps \leq \eps_0'$,
\begin{equation*}
    H\left(\theta, 2c_p\frac{\alpha\beta^2\lipc^2}{\log\eps^{-1}} \right) \leq C'_0
\end{equation*}
for some $C'_0 = C'_0(\alpha, c_p, \beta \lipc)>0$. 
By \eqref{eq:boundofK}, there exists a uniform $C' = C'(\alpha, c_p, \beta \lipc)>0$ such that
\begin{equation}
    \mathcal{K}(\theta,\eta)\leq \frac{ C'}{\sqrt{\log \eps^{-1}}}G_{\theta}(\eta).\label{eq:realboundofK} \nonumber
\end{equation}

We may choose $\alpha>1$ according to $p$ and $\beta\lipc$. Condition \eqref{eq:conditionforbeta} is valid if and only if 
\begin{equation*}
    \beta < \beta_0(p):= \frac{\sqrt{2\pi}}{2\sqrt{c_p}\lipc}.
\end{equation*}
In particular, $\beta_0(2) =\frac{1}{2} \frac{\sqrt{2\pi}}{\lipc}$ and $\beta_0(4) = \frac{1}{2\sqrt{6}}\frac{\sqrt{2\pi}}{\lipc}$.
\end{proof}

\begin{remark}\label{re:threshold}
In Section~\ref{se:mainproof} below, we only need the results of Lemma~\ref{le:Vmomentbound} and Lemma~\ref{le:MallivianMomentbound} for $p = 2, 4$. As a result, we can let $\beta_0  = \frac{1}{2\sqrt{6}}\frac{\sqrt{2\pi}}{\lipc}$ in Theorem~\ref{thm:mainthm}. Compared with the linear SHE case, we don't believe that it is an optimal value. The main difficulty of extension came from the restriction of $\beta$ in Lemma~\ref{le:MallivianMomentbound}. The coefficient $\frac{1}{\sqrt{6}}$ came from the BDG-inequality and $\frac{1}{2}$ came from \eqref{eq:boundofK}. One may find elementary ways to improve the BDG-inequality here, but while we believe that \eqref{eq:boundofK} is not sharp, it remains unclear how we can improve it.
\end{remark}

\subsection{Stein's method}
The following propositions are key ingredients in our proof of the central limit theorems Theorem~\ref{thm:mainthm} and Theorem~\ref{thm:functionalCLT}. They are derived from the Stein's method for normal approximations and are also used in \cite{gu2020nonlinear} and \cite{huang2020central}. A more detailed discussion on the Stein's method with Malliavin calculus should be referred to the monograph \cite{nourdin2012normal}.
\begin{proposition}
Let $X$ be a random variable such that $X=\delta(v)$ for $v \in \mathrm{Dom}\delta$. Assume $X \in \malD^{1,2}$. Let $Z$ be a centered Gaussian random variable with variance $\Sigma$. For any $C^2$-function $h:\R \to \R$ with a bounded second order derivative, 
\begin{equation}
    |\mathbb{E}h(X)-\mathbb{E}h(Z)|\leq \frac{1}{2}\|h''\|_{\infty}\sqrt{\mathbb{E}\left|\Sigma - \langle DX, v\rangle_{\hilb}\right|^2}. \nonumber
\end{equation}
\label{pr:SteinsMethod}
\end{proposition}
\begin{proof}
This is a special case of the Proposition \ref{pr:SteinsMethodmulti} below with $m=1$.
\end{proof}
\begin{remark}
The class of $C^2$ functions with bounded second order derivatives is a sufficient class of test functions $h(\cdot)$ to show convergence in law. With Proposition \ref{pr:SteinsMethod}, if we can prove that $\mathbb{E}\left|\Sigma - \langle DX_n, v_n\rangle_{\hilb}\right|^2 \to 0$ as $n \to \infty$ for a series of random variables $X_n \in \malD^{1,2}$ with $X_n = \delta(v_n)$, then we have $X_n \Rightarrow Z$ in law.
\end{remark}

To prove Theorem~\ref{thm:functionalCLT}, we need the following multivariate version.

\begin{proposition}
Let $F=\left(F^{(1)}, \dots, F^{(m)}\right)$ be a random vector such that $F^{(i)}=\delta(v_i)$ for $v_i \in \mathrm{Dom}\delta$, $i=1,\dots,m$. Assume $F^{(i)} \in \malD^{1,2}$ for $i=1,\dots,m$. Let $Z$ be an m-dimensional centered Gaussian vector with covariance matrix $(C_{i,j})_{1\leq i,j\leq m}$. For any $C^2$-function $h:\R ^m \to \R$ with bounded second order derivatives, 
\begin{equation}
    |\mathbb{E}h(F)-\mathbb{E}h(Z)|\leq \frac{1}{2}\|h''\|_{\infty}\sqrt{\sum_{i,j=1}^m\mathbb{E}\left|C_{i,j} - \langle DF^{(i)}, v^{(j)}\rangle_{\hilb}\right|^2}, \nonumber
\end{equation}
with
\begin{equation}
\|h''\|_{\infty} = \max_{1\leq i,j\leq m} \sup_{x\in \R^m}\left|\frac{\partial^2 h}{\partial x_i \partial x_j}(x)\right|.
    \nonumber
\end{equation}
\begin{proof}
See \cite[Proposition 2.3]{huang2020central}.
\end{proof}
\label{pr:SteinsMethodmulti}
\end{proposition}
\subsection{Local statistics of the limit}
We need the following local statistics of the limiting random field $u^{\eps}(t,x)$ as $\eps \to 0$. 
\begin{proposition}
If $\beta \lipc < \sqrt{2\pi}$, for any $t>0$ and $x \in \R^2$, we have
\begin{equation}
    u^{\eps}(t,x) \stackrel{\text{law}}{\to} \Xi_{1,2}(2) \qquad \text{as }\eps \to 0. \nonumber
\end{equation}
The random process $\Xi_{1,2}(\cdot)$ is defined as in \eqref{eq:fbsde}--\eqref{eq:fbsdeinitialcon}.
\label{pr:limitObject}
\end{proposition}
\begin{proof}
See \cite[Theorem 1.2]{dunlap2020forward}.
\end{proof}

We also need the following spatial regularity statement for $u^{\eps}(t,\cdot)$.
\begin{proposition}
For fixed $T>0$, we have
\begin{equation}
    \lim_{\eps \to 0} \sup_{\substack{t \in [0,T]\\ x_1, x_2 \in \R^2}} \frac{\left(\mathbb{E}(u^{\eps}(t,x_1)- u^{\eps}(t,x_2))^2\right)^{1/2}}{1+\eps^{-1}|x_1-x_2|}=0.\nonumber
\end{equation}\label{pr:limitreg}
\end{proposition}
\begin{proof}
See \cite[Corollary 7.2]{dunlap2020forward}.
\end{proof}

\section{Proof of Theorem~\ref{thm:mainthm}}
\label{se:mainproof}
For fixed $T>0$ and $g\in C_c^\infty(\R^2)$, let 
\begin{equation}
    v_{\eps,T}(g)(s,y) = \beta {\1}_{[0,\frac{T}{\eps^2}]}(s) \sigma\left(\bV^{\eps}(s, y)\right) \int_{\R^2}G_{\frac{T}{\eps^2}-s}(\frac{x}{\eps}-y)g(x) \dd x. \nonumber
    \label{eq:mallivianv}
\end{equation}
By the formulation of $Y^{\eps, T}$ in \eqref{eq:bigY} and \cite[Proposition 1.3.8]{nualart2006malliavin}, we have 
\begin{equation}
\bY^{\eps,T}(g) = \delta(v_{\eps,T}(g)) \in {\malD}^{1,2},
\label{eq:deltaforY}\nonumber
\end{equation}
and
\begin{align}
\label{eq:mallivianbigY}
    &\quad D_{s,y}\left(\bY^{\eps,T}(g)\right)\\&= v_{\eps,T}(g)(s,y) \nonumber\\&+  \beta \int_{s}^{\frac{T}{\eps^2}}\int_{\R^2} \int_{\R^2} G_{\frac{T}{\eps^2}-r}(\frac{x}{\eps}-z)g(x) \dd x D_{s,y}\sigma\left(\bV^{\eps}(r,z)\right) \dd W_{\phi}(r, z)\nonumber\\
    &= v_{\eps,T}(g)(s,y) \nonumber\\&+  \beta \int_{s}^{\frac{T}{\eps^2}}\int_{\R^2} \int_{\R^2} G_{\frac{T}{\eps^2}-r}(\frac{x}{\eps}-z)g(x) \dd x \Sigma(r,z)D_{s,y}\bV^{\eps}(r,z) \dd W_{\phi}(r, z).\nonumber 
\end{align}
Here $\Sigma(r,z)$ is the same random process as in \eqref{eq:malDofbigV}.

Our goal is to prove that $Y^{\eps,T}(g)$ converges in law to $\bX^T(g)$ as defined in \eqref{eq:bigXlimit} when $\eps \to 0$. Since $U$ is the random distribution that solves the Edwards-Wilkinson equation \eqref{eq:edwardswilkinson},
we have that $\bX^T(g)$ is of normal distribution $N(0,\Sigma_g^T)$ where
\begin{align}
\label{eq:Varoflimit}
\Sigma_g^T &:= \Var \left(\int_{\R^2}U(T,x)g(x)\dd x\right) \\
&= \beta^2 \mathbb{E} \sigma\left(\Xi_{1, 2}(2)\right)^{2} \int_0^T\int_{\R^{2\times 3}}G_{T-s}(x-z)G_{T-s}(y-z)g(x)g(y)\dd x\dd y \dd z \dd s \nonumber\\
&= \beta^2 \mathbb{E} \sigma\left(\Xi_{1, 2}(2)\right)^{2} \int_0^T\int_{\R^2}\int_{\R^2}G_{2(T-s)}(x-y)g(x)g(y)\dd x\dd y  \dd s \nonumber\\
&=\beta^2 \mathbb{E}\sigma(\Xi_{1,2}(2))^2 \int_0^T \int_{\R^2}|G_{T-s}*g(y)|^2\dd y\dd s\nonumber.
\end{align}
Note that the last integral is finite for any $g\in C_c^{\infty}(\R^2)$, but is not finite if $g=\delta_0$.

By Proposition~\ref{pr:SteinsMethod}, the proof of Theorem~\ref{thm:mainthm} is reduced to showing that
\begin{equation}
\mathbb{E}\left|\Sigma_g^T-\langle DY^{\eps,T}(g),v_{\eps,T}(g)\rangle_{\hilb}\right|^2 \to 0, \qquad \text{ as }\eps \to 0.
    \label{eq:stiensMethodcond}
\end{equation}
We follow the procedure in \cite[Section 4]{gu2020nonlinear}. From \eqref{eq:mallivianbigY}, we have 
\begin{equation}
\langle DY^{\eps,T}(g),v_{\eps,T}(g)\rangle_{\hilb}  = A_{1,\eps}+A_{2,\eps},
\nonumber
\end{equation}
where
\begin{align}
A_{1,\eps} &= \langle v_{\eps,T}(g),v_{\eps,T}(g)\rangle_{\hilb} \\&= \beta^2 \int_0^{\frac{T}{\eps^2}}\int_{\R^2}\int_{\R^2}\sigma(\bV^\eps(s,y_1))\sigma(\bV^\eps(s,y_2))R(y_1-y_2)\nonumber\\&\cdot \left(\int_{\R^2}G_{\frac{T}{\eps^2}-s}(\frac{x_1}{\eps}-y_1)g(x_1) \dd x_1\right)\left(\int_{\R^2}G_{\frac{T}{\eps^2}-s}(\frac{x_2}{\eps}-y_2)g(x_2) \dd x_2\right)\dd y_1 \dd y_2 \dd s.\nonumber
    \label{eq:1stTerm}
\end{align}
and
\begin{align}
A_{2,\eps} &= \langle DY^{\eps,T}(g)-v_{\eps,T}(g) ,v_{\eps,T}(g)\rangle_{\hilb} \\
&= \beta^2 \int_0^{\frac{T}{\eps^2}}\int_{\R^{2\times 2}}  \sigma(\bV^\eps(s,y_1)) \left(\int_{\R^2}G_{\frac{T}{\eps^2}-s}(\frac{x_1}{\eps}-y_1)g(x_1) \dd x_1\right)\nonumber\\&\cdot \left(\int_{s}^{\frac{T}{\eps^2}}\int_{\R^2}\left( \int_{\R^2}  G_{\frac{T}{\eps^2}-r}(\frac{x_2}{\eps}-z)g(x_2) \dd x_2\right) \Sigma(r,z)D_{s,y_2}\bV^{\eps}(r,z) \dd W_{\phi}(r, z) \right)\nonumber\\&\cdot R(y_1-y_2)\dd y_1 \dd y_2 \dd s.\nonumber
        \label{eq:2ndTerm}
\end{align}

We notice that 
\begin{equation}
    \mathbb{E}\left|\Sigma_g^T-\langle DY^{\eps,T}(g),v_{\eps,T}(g)\rangle_{\hilb}\right|^2 \leq 2\mathbb{E}\left|\Sigma_g^T-A_{1,\eps} \right|^2+ 2\mathbb{E}\left|A_{2,\eps}\right|^2. 
\end{equation}
Thus it is sufficient to prove the following two lemmas.

\begin{lemma}\label{le:sufficientlemma1}
If $\beta < \beta_0$, \begin{equation}\mathbb{E}\left|\Sigma_g^T-A_{1,\eps} \right|^2 \to 0 \qquad\text{ as }\eps \to 0.\nonumber\end{equation}
\end{lemma}
\begin{lemma}\label{le:sufficientlemma2}
If $\beta < \beta_0$, \begin{equation}\mathbb{E}\left|A_{2,\eps} \right|^2 \to 0 \qquad \text{ as }\eps \to 0.\nonumber\end{equation}
\end{lemma}

To simplify the equations, we use the convolution notation
\begin{equation}
    G_t*g(x) := \int_{\R^2}G_t(x-y)g(y)\dd y, \quad \text{for all }(t,x) \in (0,+\infty) \times \R^2.\label{heatsemidef}
    \nonumber
\end{equation}
Then by the properties of heat kernel, we have for all $t>0$,
\begin{align}
&G_t*g \in C^{\infty}(\R^2);\nonumber\\
&\|G_t*g\|_{L^{\infty}(\R^2)}\leq \|g\|_{L^{\infty}(\R^2)};\label{eq:linfinitybound}\\
&\|G_t*g\|_{L^1(\R^2)}\leq \|g\|_{L^{1}(\R^2)};\label{eq:l1bound}
\end{align}
and for all $x \in \R^2$,
\begin{equation}
    \lim_{t\to 0}G_t*g(x)=g(x).\nonumber\label{eq:heatkernellimit}
\end{equation}

We may rewrite
\begin{equation}
\int_{\R^2}G_{\frac{T}{\eps^2}-s}(\frac{x}{\eps}-y)g(x)\dd x = \eps^2 \int_{\R^2}G_{T-\eps^2s}(x-\eps y)g(x)\dd x = \eps^2 G_{T-\eps^2s}*g(\eps y).
    \nonumber \label{eq:heatsemirep}
\end{equation}
By a change of variable $\eps^2s \mapsto s$, $\eps y_1 \mapsto y_1$, $y_2-y_1 \mapsto y_2$, we have
\begin{align}
A_{1,\eps} &= \beta^2 \int_0^T\int_{\R^2}\int_{\R^2}\sigma\left(\bV^\eps(\frac{s}{\eps^2},\frac{y_1}{\eps})\right)\sigma\left(\bV^\eps(\frac{s}{\eps^2},\frac{y_1}{\eps}+y_2)\right)R(y_2) \label{eq:covAfirst}\\&\cdot G_{T-s}*g(y_1)G_{T-s}*g(y_1+\eps y_2)\dd y_1 \dd y_2 \dd s,\nonumber
\end{align}
and
\begin{align}
&A_{2,\eps} = \beta^2 \int_0^{T}\int_{\R^{2\times 2}}  \sigma\left(\bV^\eps(\frac{s}{\eps^2},\frac{y_1}{\eps})\right) \cdot \bigg(\int_{\frac{s}{\eps^2}}^{\frac{T}{\eps^2}}\int_{\R^2}
G_{T-s}*g(y_1)R(y_2) \label{eq:covAsecond}
\\&
G_{T-\eps^2r}*g(\eps z)
\Sigma(r,z)D_{\frac{s}{\eps^2},\frac{y_1}{\eps}+y_2}\bV^{\eps}(r,z) \dd W_{\phi}(r, z) \bigg)
\dd y_1 \dd y_2 \dd s.\nonumber
\end{align}

Without loss of generality, we assume that $g\geq 0$ when we estimate the integrals involving $g$ in the part (ii) of the proof of Lemma~\ref{le:sufficientlemma1} and the proof of Lemma~\ref{le:sufficientlemma2}. Observing that $|G_t*g(x)|\leq G_t*|g|(x)$ for any $t>0$ and $x \in \R^2$, one can easily check for general $g \in C_c^{\infty}(\R^2)$.

\begin{proof}[Proof of Lemma~\ref{le:sufficientlemma1}]
To prove $\mathbb{E}\left|\Sigma_g^T-A_{1,\eps} \right|^2 \to 0$ as $\eps \to 0$, we show that (i) $\mathbb{E}(A_{1,\eps})\to \Sigma_g^T$ as $\eps \to 0$, and (ii) $\Var(A_{1,\eps})\to 0$ as $\eps \to 0$.

Part (i): By \eqref{eq:covAfirst}, 
\begin{align}
    \mathbb{E}(A_{1,\eps}) &= \beta^2 \int_0^T\int_{\R^2}\int_{\R^2}\mathbb{E}\left(\sigma\left(\bV^\eps(\frac{s}{\eps^2},\frac{y_1}{\eps})\right)\sigma\left(\bV^\eps(\frac{s}{\eps^2},\frac{y_1}{\eps}+y_2)\right)\right)R(y_2) \label{eq:expectationA1}\\&\cdot G_{T-s}*g(y_1)G_{T-s}*g(y_1+\eps y_2)\dd y_1 \dd y_2 \dd s.\nonumber
\end{align}
By Proposition~\ref{pr:limitObject},  \ref{pr:limitreg} and Lemma~\ref{le:Vmomentbound}, we have that, for any $s \in (0,T)$ and $y_1,y_2 \in \R^2$, 
\begin{equation}
\mathbb{E}\left(\sigma\left(\bV^\eps(\frac{s}{\eps^2},\frac{y_1}{\eps})\right)\sigma\left(\bV^\eps(\frac{s}{\eps^2},\frac{y_1}{\eps}+y_2)\right)\right) \to \mathbb{E}\sigma(\Xi_{1,2}(2))^2, \quad \text{as }\eps \to 0.\nonumber
\end{equation}
The Lipschitz condition and Lemma~\ref{le:Vmomentbound} gives the uniform integrability to pass to the limit in \eqref{eq:expectationA1} and conclude that
\begin{align}
    \mathbb{E}(A_{1,\eps}) &\to \beta^2 \int_0^T\int_{\R^4} \mathbb{E}\sigma(\Xi_{1,2}(2))^2
    R(y_2) \cdot |G_{T-s}*g(y_1)|^2 \dd y_1\dd y_2 \dd s\nonumber\\
    &=\beta^2 \mathbb{E}\sigma(\Xi_{1,2}(2))^2 \int_0^T \int_{\R^2}|G_{T-s}*g(y_1)|^2\dd y_1\dd s\nonumber\\
    &= \Sigma_g^T, \qquad \text{as } \eps \to 0. \nonumber
\end{align}

Part (ii):
As in \cite[Lemma 4.1]{gu2020nonlinear}, let \begin{equation}
    \Lambda_{\eps}(s,y_1,y_2) = \sigma\left(\bV^\eps(\frac{s}{\eps^2},\frac{y_1}{\eps})\right)\sigma\left(\bV^\eps(\frac{s}{\eps^2},\frac{y_1}{\eps}+y_2)\right).\nonumber
\end{equation}
Then by the chain rule, for all $0\leq r < \frac{s}{\eps^2}$ and $z\in \R^2$,
\begin{align}
D_{r,z}\Lambda_{\eps}(s,y_1,y_2)&= \Sigma\left(\frac{s}{\eps^2},\frac{y_1}{\eps}\right)D_{r,z}\bV^{\eps}(\frac{s}{\eps^2},\frac{y_1}{\eps})\sigma\left(\bV^\eps(\frac{s}{\eps^2},\frac{y_1}{\eps}+y_2)\right)
    \nonumber\\
    &+\sigma\left(\bV^\eps(\frac{s}{\eps^2},\frac{y_1}{\eps})\right)\Sigma\left(\frac{s}{\eps^2},\frac{y_1}{\eps}+y_2\right)D_{r,z}\bV^{\eps}(\frac{s}{\eps^2},\frac{y_1}{\eps}+y_2).\nonumber
\end{align}
Use H\"older inequality and apply the uniform moment bounds from Lemma~\ref{le:Vmomentbound} and \ref{le:MallivianMomentbound} (with $p=4$). We obtain
\begin{align}
\label{eq:momentboundLambda}
    &\left(\mathbb{E}\left[D_{r,z}\Lambda_{\eps}(s,y_1,y_2)\right]^2\right)^{\frac{1}{2}}\\
    &\leq \frac{C}{\sqrt{\log \eps^{-1}}} \left(G_{\frac{s}{\eps^2}-r}(\frac{y_1}{\eps}-z)+G_{\frac{s}{\eps^2}-r}(\frac{y_1}{\eps}+y_2-z)\right) \nonumber
\end{align}

Now, by \eqref{eq:covAfirst} and using that 
$
    \sqrt{\operatorname{Var}\left(\int_{0}^{t} \Phi_{s} 
    \dd s\right)} \leq \int_{0}^{t} \sqrt{\operatorname{Var}\left(\Phi_{s}\right)} \dd s
$ for any process $\Phi=\{\Phi(s), s \in[0, t]\}$ with $\sqrt{\operatorname{Var}\left(\Phi_{s}\right)}$ being integrable on $[0, t]$,
we obtain
\begin{align}
    \sqrt{\Var(A_{1,\eps})} &\leq \beta^2 \int_0^T\bigg(\int_{\R^8} \Cov\left[\Lambda_{\eps}(s,y_1,y_2), \Lambda_{\eps}(s,y_1',y_2')\right]R(y_2)R(y_2')  \label{eq:varA1repr}\\
    & \cdot G_{T-s}*g(y_1)G_{T-s}*g(y_1+\eps y_2)G_{T-s}*g(y_1')G_{T-s}*g(y_1'+\eps y_2') \nonumber\\
    &\dd y_1 \dd y_1' \dd y_2 \dd y_2'\bigg)^{\frac{1}{2}}\dd s . \nonumber
\end{align}
By the Clark-Ocone formula (Proposition~\ref{pr:clark}),
\begin{equation}
\Lambda_{\eps}(s,y_1,y_2) = \mathbb{E}\left(\Lambda_{\eps}(s,y_1,y_2)\right) + \int_0^{\frac{s}{\eps^2}}\int_{\R^2}\mathbb{E}\left(D_{r,z}\Lambda_{\eps}(s,y_1,y_2) |\filt_r \right)\dd W_{\phi}(r,z).
    \nonumber
\end{equation}
Using the H\"older's inequality and the Jensen's inequality, we have
\begin{align}
&\left|\Cov\left[\Lambda_{\eps}(s,y_1,y_2), \Lambda_{\eps}(s,y_1',y_2')\right]\right|\nonumber\\
&=\int_0^{\frac{s}{\eps^2}}\int_{\R^4}\mathbb{E}\left|\mathbb{E}\left(D_{r,z}\Lambda_{\eps}(s,y_1,y_2) |\filt_r \right)\mathbb{E}\left(D_{r,z'}\Lambda_{\eps}(s,y_1',y_2') |\filt_r \right)\right|\nonumber\\
&\cdot R(z-z')\dd z \dd z' \dd r\nonumber\\
& \leq \int_0^{\frac{s}{\eps^2}}\int_{\R^4}\left(\mathbb{E}\left[D_{r,z}\Lambda_{\eps}(s,y_1,y_2) \right]^{2}\right)^{\frac{1}{2}}\left(\mathbb{E}\left[D_{r,z'}\Lambda_{\eps}(s,y_1',y_2')\right]^2\right)^{\frac{1}{2}}\nonumber\\
&\cdot R(z-z')\dd z \dd z' \dd r
    \nonumber\\
    & \leq \frac{C^2}{\log \eps^{-1}}\int_0^{\frac{s}{\eps^2}}\int_{\R^4} \left(G_{\frac{s}{\eps^2}-r}(\frac{y_1}{\eps}-z)+G_{\frac{s}{\eps^2}-r}(\frac{y_1}{\eps}+y_2-z)\right)\nonumber\\ &\cdot\left(G_{\frac{s}{\eps^2}-r}(\frac{y_1'}{\eps}-z')+G_{\frac{s}{\eps^2}-r}(\frac{y_1'}{\eps}+y_2'-z')\right)R(z-z')\dd z \dd z' \dd r.\nonumber\\
    & = \frac{C^2}{\log \eps^{-1}}\int_0^{s}\int_{\R^4} \left(G_{s-r}(y_1-z)+G_{s-r}(y_1+\eps y_2 - z)\right)\nonumber\\ &\cdot\left(G_{s-r }(y_1'-z')+G_{s-r}(y_1'+\eps y_2'- z')\right)\frac{1}{\eps^2}R(\frac{z-z'}{\eps})\dd z \dd z' \dd r.\nonumber
\end{align}
We applied \eqref{eq:momentboundLambda} and do a change of variable in the last two steps.

If we do a change of variable $z-z'\mapsto z$ and integrate in $z'$, the above expression equals
\begin{align}
      & \frac{C^2}{\log \eps^{-1}}\int_0^{s}\int_{\R^2}
      G_{2s-2r}(y_1-y_1'-z)\frac{1}{\eps^2}R(\frac{z}{\eps})+G_{2s-2r}(y_1+\eps y_2-y_1'-z)\frac{1}{\eps^2}R(\frac{z}{\eps})\\
      &+G_{2s-2r}(y_1-y_1'-\eps y_2'-z)\frac{1}{\eps^2}R(\frac{z}{\eps})+ G_{2s-2r}(y_1+\eps y_2-y_1'-\eps y_2'-z)\frac{1}{\eps^2}R(\frac{z}{\eps})\dd z \dd r.\nonumber
\end{align}
Denote $R^{\eps}(\cdot):=\frac{1}{\eps^2}R(\frac{\cdot}{\eps})$. We have $\|R^{\eps}\|_{L^1(\R^2)}=1$ and for all $t> 0$, it holds that
\begin{equation}
  \left\|G_{t}*g*R^{\eps}\right\|_{L^{\infty}(\R^2)} \leq \left\|g*R^{\eps}\right\|_{L^{\infty}(\R^2)} \leq \left\|g\right\|_{L^{\infty}(\R^2)}, \label{eq:covoinftybound}
\end{equation}
and
\begin{equation}
      \left\|G_{t}*g*R^{\eps}\right\|_{L^{1}(\R^2)} \leq \left\|g*R^{\eps}\right\|_{L^1(\R^2)} \leq \left\|g\right\|_{L^{1}(\R^2)}. \label{eq:covo1bound}
\end{equation}
Go back to \eqref{eq:varA1repr} with a change of variable $\eps y_2 \mapsto y_2$ and $\eps y_2' \mapsto y_2'$, we have 
\begin{align}
    &\sqrt{\Var(A_{1,\eps})} \nonumber\\ &\leq \frac{\beta^2C}{\sqrt{\log \eps^{-1}}} \int_0^T\bigg(\int_{\R^8} \int_0^s\big(G_{2s-2r} * R^{\eps}(y_1-y_1')  \nonumber\\
    &+G_{2s-2r} * R^{\eps}(y_1+y_2-y_1')+G_{2s-2r} * R^{\eps}(y_1-y_1'-y_2')\nonumber
   \\
   &+G_{2s-2r}*R^{\eps}(y_1+y_2-y_1'-y_2')\big)G_{T-s}*g(y_1)G_{T-s}*g(y_1+ y_2)\nonumber\\
    & \cdot G_{T-s}*g(y_1')G_{T-s}*g(y_1'+ y_2') R^{\eps}(y_2)R^{\eps}(y_2') \dd r\dd y_1 \dd y_1' \dd y_2 \dd y_2'\bigg)^{\frac{1}{2}}\dd s  \nonumber\\
    &\leq \frac{\beta^2C}{\sqrt{\log \eps^{-1}}} \int_0^T\bigg(\int_0^s\int_{\R^8} \big(G_{2s-2r} * R^{\eps}(y_1-y_1')  
    \nonumber\\&+G_{2s-2r} * R^{\eps}(y_1+y_2-y_1')+G_{2s-2r} * R^{\eps}(y_1-y_1'-y_2')
   \nonumber\\&+G_{2s-2r}*R^{\eps}(y_1+y_2-y_1'-y_2')\big)G_{T-s}*g(y_1) \cdot R^{\eps}(y_2)R^{\eps}(y_2') \dd y_1 \dd y_1' \dd y_2 \dd y_2'\dd r\bigg)^{\frac{1}{2}}\dd s 
    \nonumber\\
    & = \frac{\beta^2C}{\sqrt{\log \eps^{-1}}} \int_0^T\bigg(\int_0^s\int_{\R^6} \big(G_{T+s-2r} * R^{\eps}*g(y_1')  +G_{T+s-2r} * R^{\eps}*g(y_1'-y_2) \nonumber\\&+G_{T+s-2r} * R^{\eps}*g(y_1'+y_2')+G_{T+s-2r}*R^{\eps}*g(y_1'+y_2'-y_2)\big) \nonumber\\&\cdot R^{\eps}(y_2)R^{\eps}(y_2') \dd y_1' \dd y_2 \dd y_2'\dd r\bigg)^{\frac{1}{2}}\dd s 
    \nonumber
\end{align}
If we 
integrate each term in the order of $y_1$, $y_1'$, $y_2$, $y_2'$ and use the inequalities \eqref{eq:linfinitybound} and \eqref{eq:covo1bound}, the above is bounded by 
\begin{align}
       \frac{\beta^2C}{\sqrt{\log \eps^{-1}}}  \int_0^T\left( \int_0^s \int_{\R^4}R^{\eps}(y_2)R^{\eps}(y_2')\dd y_2 \dd y_2'\dd r\right)^{\frac{1}{2}}\dd s \leq \frac{\beta^2C}{\sqrt{\log \eps^{-1}}} \nonumber
\end{align}
The constant $C$ may vary at each occurrence. The last $C$ depend on $T$ and $\beta \lipc$.
\end{proof}

\begin{proof}[Proof of Lemma~\ref{le:sufficientlemma2}]
By Minkowski inequality and \eqref{eq:covAsecond}, 
\begin{align}
&\label{eq:A2l2bound}
    \left(\mathbb{E}|A_{2,\eps}|^2\right)^{\frac{1}{2}} \\&\leq \beta^2 \int_0^{T} \bigg(\mathbb{E}\bigg|\int_{\R^{2\times 2}}\sigma(\bV^\eps(\frac{s}{\eps^2},\frac{y_1}{\eps})) 
G_{T-s}*g(y_1)R(y_2) \nonumber\\
&\cdot \int_{\frac{s}{\eps^2}}^{\frac{T}{\eps^2}}\int_{\R^2}
G_{T-\eps^2r}*g(\eps z)
\Sigma(r,z)D_{\frac{s}{\eps^2},\frac{y_1}{\eps}+y_2}\bV^{\eps}(r,z) \dd W_{\phi}(r, z)\dd y_1 \dd y_2 \bigg|^2\bigg)^{\frac{1}{2}} \dd s\nonumber\\
&= \beta^2 \int_0^T \bigg( \int_{\R^{2\times 2}}\int_{\R^{2\times 2}} G_{T-s}*g(y_1)R(y_2) G_{T-s}*g(y_1')R(y_2')\nonumber\\
&\cdot\mathbb{E}\bigg[ \sigma(\bV^{\eps}(\frac{s}{\eps^2},\frac{y_1}{\eps}))
\sigma(\bV^{\eps}(\frac{s}{\eps^2},\frac{y_1'}{\eps}))
\nonumber\\
&\cdot\int_{\frac{s}{\eps^2}}^{\frac{T}{\eps^2}}\int_{\R^2}
G_{T-\eps^2r}*g(\eps z)
\Sigma(r,z)D_{\frac{s}{\eps^2},\frac{y_1}{\eps}+y_2}\bV^{\eps}(r,z) \dd W_{\phi}(r, z)
\nonumber\\
&\cdot\int_{\frac{s}{\eps^2}}^{\frac{T}{\eps^2}}\int_{\R^2}
G_{T-\eps^2r'}*g(\eps z')
\Sigma(r',z')D_{\frac{s}{\eps^2},\frac{y_1'}{\eps}+y_2'}\bV^{\eps}(r',z') \dd W_{\phi}(r', z')\bigg]\nonumber\\&\dd y_1 \dd y_1' \dd y_2\dd y_2' \bigg)^{\frac{1}{2}}\dd s.
\nonumber
\end{align}
Using Itô isometry, we have
\begin{align}
    &\mathbb{E}\bigg[ \sigma(\bV^{\eps}(\frac{s}{\eps^2},\frac{y_1}{\eps}))
\sigma(\bV^{\eps}(\frac{s}{\eps^2},\frac{y_1'}{\eps}))
\nonumber\\
&\cdot\int_{\frac{s}{\eps^2}}^{\frac{T}{\eps^2}}\int_{\R^2}
G_{T-\eps^2r}*g(\eps z)
\Sigma(r,z)D_{\frac{s}{\eps^2},\frac{y_1}{\eps}+y_2}\bV^{\eps}(r,z) \dd W_{\phi}(r, z)
\nonumber\\
&\cdot\int_{\frac{s}{\eps^2}}^{\frac{T}{\eps^2}}\int_{\R^2}
G_{T-\eps^2r'}*g(\eps z')
\Sigma(r',z')D_{\frac{s}{\eps^2},\frac{y_1'}{\eps}+y_2'}\bV^{\eps}(r',z') \dd W_{\phi}(r', z')\bigg]\nonumber\\
&= \int_{\frac{s}{\eps^2}}^{\frac{T}{\eps^2}}\int_{\R^2}\int_{\R^2}R(z-z')G_{T-\eps^2r}*g(\eps z)G_{T-\eps^2r}*g(\eps z')\mathbb{E}\bigg[\sigma(\bV^{\eps}(\frac{s}{\eps^2},\frac{y_1}{\eps}))
\nonumber\\
&
\cdot\sigma(\bV^{\eps}(\frac{s}{\eps^2},\frac{y_1'}{\eps}))\Sigma(r,z)D_{\frac{s}{\eps^2},\frac{y_1}{\eps}+y_2}\bV^{\eps}(r,z) 
\Sigma(r,z')D_{\frac{s}{\eps^2},\frac{y_1'}{\eps}+y_2'}\bV^{\eps}(r,z')\bigg]\dd z \dd z' \dd r.\nonumber
\end{align}
By applying \eqref{eq:bigVmobd}, Lemma~ \ref{le:MallivianMomentbound} (with $p=4$) and the Cauchy-Schwarz inequality, there exists a uniform $C>0$ such that when $\eps$ is small enough,
\begin{align}
    &\mathbb{E}\bigg[\sigma(\bV^{\eps}(\frac{s}{\eps^2},\frac{y_1}{\eps}))
\sigma(\bV^{\eps}(\frac{s}{\eps^2},\frac{y_1'}{\eps}))
\Sigma(r,z)D_{\frac{s}{\eps^2},\frac{y_1}{\eps}+y_2}\bV^{\eps}(r,z) 
\Sigma(r,z')D_{\frac{s}{\eps^2},\frac{y_1'}{\eps}+y_2'}\bV^{\eps}(r,z')\bigg]\nonumber\\
&\leq \frac{C^2}{\log \eps^{-1}} G_{r-\frac{s}{\eps^2}}(z-\frac{y_1}{\eps}-y_2)G_{r-\frac{s}{\eps^2}}(z'-\frac{y_1'}{\eps}-y_2').\nonumber
\end{align}
Substitute back into \eqref{eq:A2l2bound} and integrate in $y_1$, $y_1'$, we have
\begin{align}
\label{eq:A2l2Boundsim}
    &\left(\mathbb{E}|A_{2,\eps}|^2\right)^{\frac{1}{2}} \leq \frac{\beta^2 C}{\sqrt{\log \eps^{-1}}} \int_0^T \bigg( \int_{\R^{2\times 2}}\int_{\R^{2\times 2}} G_{T-s}*g(y_1)R(y_2) G_{T-s}*g(y_1')R(y_2')\\
&\int_{\frac{s}{\eps^2}}^{\frac{T}{\eps^2}}\int_{\R^2}\int_{\R^2}R(z-z')G_{T-\eps^2r}*g(\eps z)G_{T-\eps^2r}*g(\eps z')\nonumber\\
&\cdot G_{r-\frac{s}{\eps^2}}(z-\frac{y_1}{\eps}-y_2)G_{r-\frac{s}{\eps^2}}(z'-\frac{y_1'}{\eps}-y_2')\dd z \dd z' \dd r
\dd y_1 \dd y_1' \dd y_2\dd y_2' \bigg)^{\frac{1}{2}}\dd s\nonumber\\
&= \frac{\beta^2 C}{\sqrt{\log \eps^{-1}}} \int_0^T\bigg(\eps^2\int_{s}^T \int_{\R^8} R(y_2)R(y_2')R(z-z')G_{T-r}*g(\eps z)G_{T-r}*g(\eps z') \nonumber\\
&\cdot G_{T+r-2s}*g(\eps z-\eps y_2)G_{T+r-2s}*g(\eps z'-\eps y_2')\dd r\dd z \dd z' \dd y_2 \dd y_2' \bigg)^{\frac{1}{2}}\dd s \nonumber\\
&= \frac{\beta^2 C}{\sqrt{\log \eps^{-1}}} \int_0^T\bigg(\int_{s}^T \int_{\R^8} R^{\eps}(y_2)R^{\eps}(y_2')R^{\eps}(z-z')G_{T-r}*g( z)G_{T-r}*g(z') \nonumber\\
&\cdot G_{T+r-2s}*g( z- y_2)G_{T+r-2s}*g( z'- y_2')\dd r\dd z \dd z' \dd y_2 \dd y_2' \bigg)^{\frac{1}{2}}\dd s. \nonumber
\end{align}
Here we do a change of variable $\eps y_2 \mapsto y_2$, $\eps y_2' \mapsto y_2'$, $\eps z\mapsto z$ and $\eps z' \mapsto z'$. We integrate the last integral of \eqref{eq:A2l2Boundsim} in the order of $y_2$, $y_2'$, $z$,$z'$. Together with \eqref{eq:l1bound} and \eqref{eq:covoinftybound}, we obtain that the last expression is bounded by 
\begin{align}
    &\frac{\beta^2C}{\sqrt{\log \eps^{-1}}} \int_0^T\bigg(\int_{s}^T\int_{\R^4} R^{\eps}(z-z')G_{T-r}*g(z)G_{T-r}*g(z')\dd r \dd z \dd z'\bigg)^{\frac{1}{2}}\dd s\nonumber\\
    & \leq \frac{\beta^2C}{\sqrt{\log \eps^{-1}}},
    \nonumber
\end{align}
The constant $C$ may vary at each occurrence. The last $C$ would depend on $T$ and $\beta \lipc$.
\end{proof}

\section{Proof of Theorem~\ref{thm:functionalCLT}}
\label{se:clt}
It suffices to prove the convergence of the finite-dimensional
distributions and the tightness. 

To give the tightness, we use the following proposition.
\begin{proposition}
For $p \geq 2$, $T>0$, $g \in C_c^{\infty}(\R^2)$, let $\beta_0(p)$ and $\eps_0$ be defined as in Lemma~\ref{le:Vmomentbound}.  If $\beta<\beta_0(p)$, there exists a uniform constant $C>0$ (depending only on $ \beta \lipc$, $p$ and $g$), such that for any $0 \leq s <t \leq T$ and $\eps < \eps_0$,
    $\mathbb{E}\left|X^{\eps, t}(g)-X^{\eps, s}(g) \right|^p \leq C(t-s)^{\frac{p}{2}}$.\label{pr:tigntness}
\end{proposition}
\begin{proof}
By the mild formulation of $u^\eps$ in \eqref{eq:umild} and the Burkholder-Davis-Gundy inequality, let $c_p = \frac{p(p-1)}{2}$, we have 
\begin{align*}
     &\left(\mathbb{E}\left|X^{\eps, t}(g)-X^{\eps, s}(g) \right|^p \right)^{\frac{2}{p}}\\&=\left( \mathbb{E}\left|\sqrt{\log \eps^{-1}}\int_{\R^2}\left[u^{\eps}(t,x)- u^{\eps}(s,x)\right]g(x)\dd x\right|^p\right)^{\frac{2}{p}}\\
     &= \bigg(\mathbb{E}\bigg|\beta\int_{\R^2}\bigg[
     \int_0^t\int_{\R^2}G_{t-r}(x-y)\sigma(u^{\eps}(r,y))\dd W_{\phi^{\eps}}(r,y) \\&- 
     \int_0^s\int_{\R^2}G_{s-r}(x-y)\sigma(u^{\eps}(r,y))\dd W_{\phi^{\eps}}(r,y)
     \bigg]g(x)\dd x\bigg|^p\bigg)^{\frac{2}{p}}\\
     & = \beta^2 \left(\mathbb{E}\bigg|
     \int_0^T\int_{\R^2}\left(G_{t-r}*g(y)\1_{[0,t]}(r)-G_{s-r}*g(y)\1_{[0,s]}(r)\right)\sigma(u^{\eps}(r,y))\dd W_{\phi^{\eps}}(r,y) \bigg|^p\right)^{\frac{2}{p}}\\
     &\leq c_p\beta^2 \int_0^T \bigg( \mathbb{E}
     \bigg(\int_{\R^2} \int_{\R^2} \left(G_{t-r}*g(y_1)\1_{[0,t]}(r)-G_{s-r}*g(y_1)\1_{[0,s]}(r)\right)\\
     &\cdot \left(G_{t-r}*g(y_2)\1_{[0,t]}(r)-G_{s-r}*g(y_2)\1_{[0,s]}(r)\right)\\
     &\cdot \left(\sigma(u^{\eps}(r,y_1))\sigma(u^{\eps}(r,y_2))\right)\frac{1}{\eps^2}R(\frac{y_1-y_2}{\eps}) \dd y_1 \dd y_2 \bigg)^{\frac{p}{2}} \bigg)^{\frac{2}{p}} \dd r.
\end{align*}
Apply the Minkowski inequality, inequality \eqref{eq:uholder} and Lemma~\ref{le:Vmomentbound}, the above is bounded by
\begin{align*}
     & c_p\beta^2 \lipc^2\int_0^T \int_{\R^2} \int_{\R^2} \left(G_{t-r}*g(y_1)\1_{[0,t]}(r)-G_{s-r}*g(y_1)\1_{[0,s]}(r)\right)\\
     &\cdot \left(G_{t-r}*g(y_2)\1_{[0,t]}(r)-G_{s-r}*g(y_2)\1_{[0,s]}(r)\right) \left(\mathbb{E}|u^{\eps}(r,0)|^p\right)^{\frac{2}{p}}\frac{1}{\eps^2}R(\frac{y_1-y_2}{\eps}) \dd y_1 \dd y_2 \dd r\\
     &\leq C\beta^2\lipc^2\int_0^T \int_{\R^2} \int_{\R^2} \left(G_{t-r}*g(y_1)\1_{[0,t]}(r)-G_{s-r}*g(y_1)\1_{[0,s]}(r)\right)\\
     &\cdot \left(G_{t-r}*g(y_2)\1_{[0,t]}(r)-G_{s-r}*g(y_2)\1_{[0,s]}(r)\right) \frac{1}{\eps^2}R(\frac{y_1-y_2}{\eps}) \dd y_1 \dd y_2 \dd r\\
     & \leq C\beta^2\lipc^2\int_0^T\left( \int_{\R^2} \int_{\R^2} \left(G_{t-r}*g(y_1)\1_{[0,t]}(r)-G_{s-r}*g(y_1)\1_{[0,s]}(r)\right)^2\frac{1}{\eps^2}R(\frac{y_1-y_2}{\eps}) \dd y_1 \dd y_2 \right)^{\frac{1}{2}}\\
     &\cdot \left( \int_{\R^2} \int_{\R^2}\left(G_{t-r}*g(y_2)\1_{[0,t]}(r)-G_{s-r}*g(y_2)\1_{[0,s]}(r)\right) \frac{1}{\eps^2}R(\frac{y_1-y_2}{\eps}) \dd y_1 \dd y_2\right)^{\frac{1}{2}} \dd r\\
     &= C\beta^2\lipc^2\int_0^T\int_{\R^2}  \left(G_{t-r}*g(y)\1_{[0,t]}(r)-G_{s-r}*g(y)\1_{[0,s]}(r)\right)^2 \dd y \dd r.
\end{align*}
Here we used the H\"older inequality and the fact that $\|R^{\eps}\|_{L^1}=1$ in the last two steps. Denote the Fourier transform of $g$ by $\hat{g}$. Adopt a similar procedure as in \cite[Proposition 4.1]{huang2020central}, we have
\begin{align*}
    &\int_0^T\int_{\R^2}  \left(G_{t-r}*g(y)\1_{[0,t]}(r)-G_{s-r}*g(y)\1_{[0,s]}(r)\right)^2 \dd y \dd r \\
    & = C\int_0^T\int_{\R^2}  \left(e^{-\frac{t-r}{2}|\xi|^2}\1_{[0,t]}(r)-e^{-\frac{s-r}{2}|\xi|^2}\1_{[0,s]}(r)\right)^2 \hat{g}(\xi)^2\dd \xi \dd r\\
    &= C \int_0^s\int_{\R^2} 
    e^{-(s-r)|\xi|^2}
    \left(e^{-\frac{t-s}{2}|\xi|^2}-1\right)^2 \hat{g}(\xi)^2 \dd \xi \dd r + C \int_s^t\int_{\R^2}  e^{-(t-r)|\xi|^2} \hat{g}(\xi)^2 \dd \xi \dd r\\
    &= C \int_{\R^2} \frac{1}{|\xi|^2}\left(1-
    e^{-s|\xi|^2}\right)
    \left(e^{-\frac{t-s}{2}|\xi|^2}-1\right)^2 \hat{g}(\xi)^2 \dd \xi 
    + C \int_{\R^2}  \frac{1}{|\xi|^2}\left(1-e^{-(t-s)|\xi|^2}\right) \hat{g}(\xi)^2 \dd \xi\\
    &\leq C \int_{\R^2} \frac{1}{|\xi|^2}\left(1-
    e^{-s|\xi|^2}\right)
 \frac{t-s}{2}|\xi|^2 \hat{g}(\xi)^2 \dd \xi + C \int_{\R^2}  \frac{1}{|\xi|^2}(t-s)|\xi|^2 \hat{g}(\xi)^2 \dd \xi\\
 &\leq C(t-s)\int_{\R^2}\hat{g}(\xi)^2\dd \xi.
\end{align*}
In the last steps, we applied  inequalities $|1-e^{-a}|\leq \sqrt{a}$ and $(1-e^{-a}) \leq a$ for all $a \geq 0$. The constant $C$ may vary at each occurrence.

In the proof, the conditions $\beta <\beta_0$ and $\eps < \eps_0$ are given only to validate the use of Lemma~\ref{le:Vmomentbound}. Therefore, we can choose $\beta_0(p)$ and $\eps_0$ to be the same value as in Lemma~\ref{le:Vmomentbound}.
\end{proof}

Choose $p=4$. By \cite[Problem 4.11]{karatzas2012brownian}, the measure induced by $X^{\eps}(g) = \{X^{\eps,t}(g); 0 \leq t \leq T\}$ on $(C([0,T]), \mathscr{B}(C([0,T])))$ is tight. By \cite[Theorem 4.15]{karatzas2012brownian}, to prove Theorem~\ref{thm:functionalCLT}, we only need to show the convergence of the finite-dimensional
distribution.

\begin{lemma}
Fix $0 \leq t_1 < \dots < t_m \leq T$. As $\eps \to 0$, the sequence of random vectors $\left(X^{\eps, t_1}(g), \dots, X^{\eps, t_m}(g)\right)$ converges in law to the $m$-dimensional Gaussian centered vector $Z=(Z^1,\dots,Z^m)$ with covariance
\begin{equation}
C_{i,j}:=\mathbb{E}(Z^{i}Z^{j})= \beta^2 \mathbb{E} \sigma\left(\Xi_{1, 2}(2)\right)^{2}  \int_0^{t_i \wedge t_j}  \int_{\R^2}G_{t_i-s}*g(y)G_{t_j-s}*g(y)\dd y\dd s.
    \nonumber\label{eq:covariance}
\end{equation}\label{le:finiteconv}
\end{lemma}
\begin{proof}
Recall that we have shown 
\begin{equation*}
    \left(X^{\eps, t_1}(g), \dots, X^{\eps, t_m}(g)\right) \stackrel{\text { law }}{=} \left(Y^{\eps, t_1}(g), \dots, Y^{\eps, t_m}(g)\right),
\end{equation*}
and 
\begin{equation*}
\bY^{\eps,t_i}(g) = \delta(v_{\eps,t_i}(g)) \in {\malD}^{1,2}.
\end{equation*}
By Proposition~\ref{pr:SteinsMethodmulti}, it suffices to show that for each $i,j$,
\begin{equation}
    \mathbb{E}|C_{i,j}- \langle DY^{\eps, t_i}(g), v_{\eps, t_j}(g) \rangle_{\hilb} |^2 \to 0. \nonumber
\end{equation}
Use \eqref{eq:mallivianbigY}, we have
\begin{align*}
     &\langle DY^{\eps, t_i}(g), v_{\eps, t_j}(g) \rangle_{\hilb} = A_{1,\eps}^{i,j}+ A_{2,\eps}^{i,j},
\end{align*}
where
\begin{align*}
    A_{1,\eps}^{i,j} &= \langle v_{\eps,t_i}(g), v_{\eps, t_j}(g) \rangle_{\hilb} \\
    &= \beta^2 \int_0^{\frac{t_i\wedge t_j}{\eps^2}}\int_{\R^2}\int_{\R^2}\sigma(\bV^\eps(s,y_1))\sigma(\bV^\eps(s,y_2))R(y_1-y_2)\nonumber\\&\cdot \left(\int_{\R^2}G_{\frac{t_i}{\eps^2}-s}(\frac{x_1}{\eps}-y_1)g(x_1) \dd x_1\right)\left(\int_{\R^2}G_{\frac{t_j}{\eps^2}-s}(\frac{x_2}{\eps}-y_2)g(x_2) \dd x_2\right)\dd y_1 \dd y_2 \dd s.\nonumber
\end{align*}
and
\begin{align*}
    A_{2,\eps}^{i,j} &=\langle DY^{\eps,t_i}(g)-v_{\eps,t_i}(g) ,v_{\eps,t_j}(g)\rangle_{\hilb} \\
&= \beta^2 \int_0^{\frac{t_i \wedge t_j}{\eps^2}}\int_{\R^{2\times 2}}  \sigma(\bV^\eps(s,y_1)) \left(\int_{\R^2}G_{\frac{t_j}{\eps^2}-s}(\frac{x_1}{\eps}-y_1)g(x_1) \dd x_1\right)\nonumber\\&\cdot \left(\int_{s}^{\frac{t_i}{\eps^2}}\int_{\R^2}\left( \int_{\R^2}  G_{\frac{t_i}{\eps^2}-r}(\frac{x_2}{\eps}-z)g(x_2) \dd x_2\right) \Sigma(r,z)D_{s,y_2}\bV^{\eps}(r,z) \dd W_{\phi}(r, z) \right)\nonumber\\&\cdot R(y_1-y_2)\dd y_1 \dd y_2 \dd s.
\end{align*}
Following the same arugments as in Section~\ref{se:mainproof}, we can show
\begin{equation*}
      \mathbb{E}\left|C_{i,j}-A_{1,\eps}^{i,j}\right|^2 \to 0 \qquad \text{as } \eps \to 0,
\end{equation*}
and
\begin{equation*}
    \mathbb{E}\left|A_{2,\eps}^{i,j}\right|^2 \to 0 \qquad \text{as } \eps \to 0.
\end{equation*}
\end{proof}

\newpage
\bibliographystyle{plain}
\bibliography{ref}
\end{document}